\documentclass[11pt,reqno]{amsart}
\usepackage[english]{babel}
\usepackage{mathtools,amsmath,amssymb}
\usepackage{enumitem}
\usepackage{mathrsfs, dsfont}
\usepackage{xcolor}
\usepackage[pdfencoding=auto, psdextra]{hyperref}
\hypersetup{colorlinks=true,allcolors=[rgb]{0,0,0.6}}
\usepackage[hmargin=2.7cm,vmargin=2.6cm]{geometry}
\usepackage{soul}

\newtheorem{theorem}{Theorem}[section]
\newtheorem{proposition}[theorem]{Proposition}
\newtheorem{lemma}[theorem]{Lemma}
\newtheorem{corollary}[theorem]{Corollary}
\theoremstyle{definition}
\newtheorem{definition}[theorem]{Definition}

\newtheorem{example}[theorem]{Example}
\theoremstyle{remark}
\newtheorem{remark}[theorem]{Remark}
\theoremstyle{definition}
\newtheorem{assumption}[theorem]{Assumption}

\numberwithin{equation}{section}

\newcommand\br{\begin{remark}}
\newcommand\er{\end{remark}}
\newcommand\bp{\begin{pmatrix}}
\newcommand\ep{\end{pmatrix}}
\newcommand{\be}{\begin{equation}}
\newcommand{\ee}{\end{equation}}
\newcommand{\ba}[1]{\begin{array}{#1}}
\newcommand{\ea}{\end{array}}

\newcommand{\beg}{\begin{example}}
\newcommand{\eeg}{\end{example}}
\newcommand{\bpr}{\begin{proposition}}
\newcommand{\epr}{\end{proposition}}
\newcommand{\bt}{\begin{theorem}}
\newcommand{\et}{\end{theorem}}
\newcommand{\bc}{\begin{corollary}}
\newcommand{\ec}{\end{corollary}}
\newcommand{\bl}{\begin{lemma}}
\newcommand{\el}{\end{lemma}}
\newcommand{\bd}{\begin{definition}}
\newcommand{\ed}{\end{definition}}
\newcommand{\brs}{\begin{remarks}}
\newcommand{\ers}{\end{remarks}}

\def\eps {\varepsilon}

\newcommand{\NN}{{\mathbb N}}

\newcommand{\QQ}{{\mathbb Q}}
\newcommand{\RR}{{\mathbb R}}
\newcommand{\TT}{{\mathbb T}}

\newcommand\cA{{\mathcal A}}

\newcommand\cF{{\mathcal F}}
\newcommand\cG{{\mathcal G}}

\newcommand\cK{{\mathcal K}}

\newcommand\cN{{\mathcal N}}

\begin{document}

\title[Almost sure  existence of global solutions ]{
Almost sure existence of global solutions for general initial value problems
}

\author{Zied Ammari}
\address{Univ Rennes, [UR1], CNRS, IRMAR - UMR 6625, F-35000 Rennes, France.}
\email{zied.ammari@univ-rennes1.fr}
\author{Shahnaz Farhat}
\address{Univ Rennes, [UR1], CNRS, IRMAR - UMR 6625, F-35000 Rennes, France.}
\email{shahnaz.farhat@univ-rennes1.fr}
\author{Vedran Sohinger }
\address{University of Warwick, Mathematics Institute, Zeeman Building, Coventry CV4 7AL, United Kindgom.}
\email{V.Sohinger@warwick.ac.uk.}

\date{July 20, 2023}


\begin{abstract}
This article is concerned with the almost sure existence of global solutions for  initial value  problems of the  form $\dot{\gamma}(t)= v(t,\gamma(t))$ on separable dual Banach spaces. We prove a general result  stating that whenever there exists $(\mu_t)_{t\in\RR}$ a family of probability  measures satisfying a related  statistical Liouville equation, there exist global solutions to the initial value problem for $\mu_0$-almost all initial data, possibly without uniqueness. The main assumption  is a mild  integrability condition of the vector field $v$ with respect to $(\mu_t)_{t\in\RR}$.  As a notable application, we obtain  from the above principle that Gibbs  and Gaussian measures   yield low regularity global solutions for several nonlinear dispersive PDEs  as well as  fluid mechanics equations including the Hartree, Klein-Gordon, NLS,  Euler  and  modified surface quasi-geostrophic  equations.  In this regard, our result generalizes Bourgain's method  \cite{MR1374420} as well as Albeverio {\tiny \&} Cruzeiro's method \cite{MR1051499} of constructing low regularity global solutions, without the need for  local well-posedness analysis.
\end{abstract}

\maketitle
\tableofcontents

\section{Introduction and main results}
\label{chapintro}

Initial value problems, including ODEs, PDEs and stochastic PDEs,   are of major interest to both applied  and fundamental Mathematics.  There is an abundant literature for this broad field of research,  covering often important evolution equations in science, see {\it e.g.}~\cite{MR0390843,MR0069338,MR1867882}.  From a theoretical point of view, one can  mainly recognize two qualitative approaches:
\begin{itemize}
\item[--] A specific analysis that relies on the exact or almost exact form of initial value problems using   particular  features of given equations ({\it e.g.}~exactly solvable equations, dispersive,  hyperbolic,  parabolic and to some extent  semilinear equations).
\item[--] A more general analysis that ignores the exact form  of initial value problems and instead focuses on finding general criteria that ensure uniqueness, local and global existence of solutions ({\it e.g.}~Carath\'eodory, Cauchy-Lipschitz, Peano theorems and to some extent fixed-point theorems).
\end{itemize}
Of course, these two perspectives complement each other. On the other hand, {there is} a sharp distinction between initial value problems over finite and infinite-dimensional spaces. For instance, in infinite dimensions  it is known  that the Peano theorem is in general not true  and that there exist finite lifespan solutions to initial value problems without blowup.  However, the Cauchy-Lipschitz theorem still holds true on Banach spaces, thus indicating that  certain results could indeed survive in infinite dimensions. The present article is concerned with the second approach.

\medskip

In the past few decades, there have been significant advances in the field of dispersive PDEs to construct almost sure global solutions with low regularity, with methods stemming from the combination of probability theory, harmonic analysis and quantum field theory. These advances were inspired by the pioneering work of Bourgain \cite{MR1309539,MR1374420,MR1470880}. The latter papers continued the line of research initiated by Lebowitz--Rose--Speer \cite{MR939505} and Zhidkov \cite{MR1121337}, as well as McKean--Vaninsky \cite{MR1441912,MR1447055,MR1277197}.  There have been many subsequent contributions on this subject, see \cite{Bringmann_2,MR4385403,Bringmann_Deng_Nahmod_Yue,MR2425134,MR4077096,MR3455152,MR4236191,MR4259382,MR4411729,MR4218790,MR3518561,MR4452511,MR4324723,MR2928851,MR2875861,MR3361727,MR4384196,MR3844655,MR4133695,MR4030277} and the works quoted there. For an overview, we also refer the reader to the expository works \cite{MR3869074,MR3952697,MR4575389} and the references therein. One of the main ideas of the aforementioned works is that Gibbs measures for some Hamiltonian PDEs are  well-defined over Sobolev spaces with sufficiently low  regularity exponent and that they are formally invariant under the flow. However, the other known conservation laws are available only at higher regularity. Consequently, the Gibbs measures can be used as a substitute for a conservation law. When combined with local well-posedness theory, they can be used to construct low regularity global solutions for almost every initial data.   On the other hand, in the field of fluid mechanics, there is a long-standing interest in constructing invariant measures and global solutions using probabilistic methods
(see {\it e.g.}~\cite{MR537263,MR996609,MR1669421}).  In particular, in the nineties, Albeverio and Cruzeiro proved the
almost sure existence of global solutions to the Euler equation on the two-dimensional torus \cite{MR1051499}.  More recently, there has been a renewed interest for such invariant measures  and stochastic flows in fluid mechanics (see {\it e.g.}~\cite{MR4493142,MR3910197,MR4031773}).
In the work of Nahmod--Pavlovi\'{c}--Staffilani--Totz \cite{MR3818404}, the authors extended the result of
Albeverio--Cruzeiro  \cite{MR1051499} to the modified  quasi-geostrophic (mSQG) equations which interpolate between the Euler and  the SQG equation.  Again the idea behind the construction of these global flows is the invariance of a given well-understood (Gaussian) measure combined with probabilistic  compactness arguments  in the spirit of Prokhorov's or Skorokhod's theorems.
Our main claim here is that the previous circle of  ideas is quite  general and robust  and  could be formulated as a general principle for abstract initial value problems.

\medskip
Our aim in this article is to address the question of almost sure existence of global solutions from a more general  perspective. More precisely, we consider  an abstract initial value problem of the form
\begin{equation}
\label{chap4.intro.eq1}
\dot\gamma(t)=v(t,\gamma(t)).
\end{equation}
over  a  separable dual Banach space\footnote{{\it i.e.} this means that there exists $E$ a Banach space such that $B$ is a topological dual of $E$ and $B$ is separable.} $(B,\Vert\cdot\Vert)$ with $v:\mathbb{R}\times B\to B$  a Borel vector field.  We assume that there exists a narrowly continuous\footnote{see Definition \ref{narrow_continuity}.}  probability measure   solution   $(\mu_t)_{t\in\RR}$ to the  statistical Liouville equation
\begin{equation}
\label{chap4.intro.eq2}
\frac{d}{dt} \int_B F(u) \,  \mu_t(du)= \int_B \langle v(t,u), \nabla F(u) \rangle \, \mu_t(du)\,,
\end{equation}
with the same vector field. Here $\langle  \cdot, \cdot \rangle$ denotes the duality bracket between the topological dual  space $B=E^*$ and its predual $E$, while the $F$ range over a class of smooth test functions with $\nabla F$ {denoting} their Fr\'echet differentials. The statistical Liouville equation  \eqref{chap4.intro.eq2} is explained in detail in Section  \ref{subsec1.1}. The considered class of smooth test functions is given in Definition \ref{cylindrical_test_functions}.  In this framework, we show that for $\mu_0$-almost all initial data in $B$  there exist global solutions  to the initial value problem \eqref{chap4.intro.eq1},  if the vector field satisfies the integrability condition
\begin{equation}
\label{chap4.intro.eq3}
t\in \RR\mapsto \int_{B}\Vert v(t,u) \Vert  \,\mu_t(du) \in L^1_{loc}(\RR,dt).
\end{equation}
Remarkably, such a result depends neither
on the shape of the vector field, nor on a suitable local well-posedness theory. Moreover,  the method applies equally  to finite or infinite-dimensional spaces and the vector field $v$ is not required to be continuous.
In practice, our result reduces the problem of constructing  global solutions of ODEs or PDEs to finding probability measure solutions for the statistical Liouville equation \eqref{chap4.intro.eq2}.
The latter problem is sometimes more tractable. In many of the cases which we study later, one can directly construct global  probability measure solutions $(\mu_t)_{t\in\RR}$ in such a way that they are either stationary or  stationary modulo a pushforward. Generally speaking, recall  that for Hamiltonian systems the  Liouville theorem ensures the existence of invariant measures while for dynamical systems the Krylov-Bogolyubov theorem  (\cite[Lecture 2]{MR0584788} and \cite{MR1417491}) is an efficient tool for constructing  invariant measures.  Note that  our understanding of invariance here is in terms of  stationary solutions for the statistical Liouville equation \eqref{chap4.intro.eq2} instead of the invariance with respect to the flow, since the latter may not exist in general.

\medskip
On the other hand, in the recent work  \cite{MR4571599}, the Kubo-Martin-Schwinger (KMS) equilibrium states were introduced for Hamiltonian PDEs (like Hartree, NLS, Wave equations). In particular, it was proved in this context that Gibbs measures  are KMS equilibrium states satisfying the stationary statistical Liouville equation \eqref{chap4.intro.eq2} with an appropriate choice of the vector field $v$ in accordance with the given PDE.   Hence, as a consequence of the above principle and the stationarity of Gibbs measures one deduces straightforwardly  the existence of low regularity global solutions for almost all initial data for several Hamiltonian PDEs.  It is also worth highlighting  here the two features of our approach:
\begin{itemize}
\item[--]  No dispersive properties are {needed}.
\item[--]  More general nonlinearities can be considered by relying on Malliavin calculus (see \cite[Chapter 1]{MR2200233}).
\end{itemize}
We refer the reader to Section \ref{chap3.application} for more details on the examples of nonlinear PDEs considered here (including Hartree, NLS and Wave equations on the flat torus $\mathbb{T}^d$, $d=1,2,3$ (see \cite{MR3869074}) and Euler, modified surface quasi-geostrophic  (mSQG) equations on the  $2$-dimensional torus (see \cite{MR1051499,MR3818404})).
Of course, we do not prove here global well-posedness  for such PDEs. Instead,  we show almost surely the existence    of global solutions for such PDEs.  Potentially, one can try to combine our method with a local posedness theory in order to prove  global well-posedness (see \cite{Burq_2008,MR2425134}).    In contrast,  uniqueness generally speaking  depends more on  the particular properties of the considered initial value problem. Another aspect that we did not address is global solutions for random systems. Instead, we focus here on deterministic equations.
In particular, it makes sense to study the problem for stochastic PDEs and random dynamical systems with widespread applications in fluid mechanics and stochastic quantization for instance (see \cite{MR3444271} and \cite{MR2016604}).
The analysis for these random systems will be addressed elsewhere.

\medskip
\emph{Techniques:} Our approach is quite related to statistical physics in spirit and consists of studying the evolution of ensembles of initial data through statistical Liouville equations.  However, the key argument comes from transport theory via the superposition principle (or probabilistic representation)  proved for instance in
\cite{AmbrosioLuigi2005GFIM,MR2400257,MR2335089} (see also \cite{MR2668627}).  The superposition principle shows in particular that if one has a  probability measure solution $(\mu_t)_{t\in[0,T]}$ to the statistical Liouville equation  \eqref{chap4.intro.eq2}   on the finite time interval $[0,T]$,  then there exists a probability path measure $\eta$ concentrated on the set of local solutions of the initial value problem \eqref{chap4.intro.eq1}    such that the image measure of $\eta$ under the evaluation map
  at each fixed time is equal to $\mu_t$ (see Proposition \ref{propinfinite}). Such a result  is extended to infinite dimensions and adapted to PDE analysis in \cite{MR3721874}.  Here, we extend such a principle to separable dual Banach spaces and more importantly to global solutions of initial value problems in such a way that  this tool yields a powerful globalisation argument. In particular,  by using the measurable projection theorem and the properties of the path measure $\eta$,  we are  able to find a universally measurable subset $\mathscr G$ of the Banach space $B$  such that $\mu_0(\mathscr G)=1$ and   such that for each $x\in \mathscr G$ there exists a global solution to the initial value problem \eqref{chap4.intro.eq1}. Furthermore, we construct a measurable flow in Theorem \ref{thm_measurable_flow} if in addition we assume that there exists at most one global mild solution of the initial value problem \eqref{chap3.ivpinh} for each initial condition.   Moreover, in Theorem \ref{thm_flow_invariance} we prove that if  the  initial value problem \eqref{chap3.ivpinh} admits a measurable flow (See Definition \ref{def.mes.flow}) then $(\mu_t)_{t\in\RR}$ satisfies an appropriate statistical Liouville equation.

\medskip
In conclusion, the globalization result (Theorem \ref{chap3.main.thm}) and Theorems \ref{thm_measurable_flow}-\ref{thm_flow_invariance} proved here by measure theoretical techniques    are   quite general and, to the best of our knowledge, new.   They formalize and unify  some  of the deep  ideas in the topic of constructing almost    sure  low regularity global solutions to dispersive PDEs and to fluid mechanics equations. As an application, we are able to recover several known results and to obtain new ones (see Sections \ref{chap3.application.ode}-\ref{chap3.application}).  The article also connects the problem of constructing global flows for PDEs with the topic of continuity and transport equations (see \cite{MR2475421,AmbrosioLuigi2005GFIM}).

\subsection{General framework}\label{subsec1.1}
Let  $(B,||\cdot||)$ be a real separable dual Banach space. This means that there exists  $(E, \|\cdot\|_E)$ a real Banach space such that $B$ is the topological dual of $E$ ({\it i.e.}\  $B=E^*$, $\|\cdot\|=\|\cdot\|_{E^*}$) and $(B,||\cdot||)$ is separable.  Recall  that since $(B,||\cdot||)$ is  separable then $(E, \|\cdot\|_E)$ is also separable (see \cite[Theorem III.23]{MR697382}).  When there is no possible confusion, we will denote the duality bracket $\langle \cdot,\cdot\rangle_{E^*,E}\,$  simply by $\langle \cdot,\cdot\rangle$.

\medskip
In all the sequel, $I$ denotes a closed unbounded time interval
({\it e.g.}~$I=\RR$, $I=\RR_-$ or $I=\RR_+$). We denote by  $t_0\in I$  any initial time if $I=\RR$. If $I$ is bounded from below or above,  then we denote by $t_0\in I$ its endpoint.
Our main purpose is the study of the \emph{initial value problem}
\begin{equation}\label{chap3.ivpinh}
\begin{cases}
\dot{\gamma}(t)=v(t,\gamma(t)), \\
\gamma(t_0)=x \in B,
\end{cases}
\end{equation}
where  $v: I\times B \to B$ is a {Borel  vector field}.  Generally speaking, there are  several  notions of solutions to  \eqref{chap3.ivpinh}.
A  \emph{strong solution} is a curve  $\gamma$ belonging to $\mathscr{C}^1(I;B)$ and satisfying  \eqref{chap3.ivpinh}  for all times $t\in I$. However, to study such curves one usually requires $v$  to be at least continuous in order to have a consistent equation.  Instead, we focus on  \emph{mild solutions}  of  \eqref{chap3.ivpinh}  which are continuous curves $ \gamma\in\mathscr{C}(I; B)$ such that $v(\cdot,\gamma(\cdot))\in L_{loc}^1(I,dt;B)$ and for all $t\in I$ the following  integral formula is satisfied.
\be \label{duhamelformula}
\gamma(t)=x+\int_{t_0}^{t} v (s,\gamma(s))  \, ds.
\ee
Here, the integration in the right hand side is a Bochner integral and the function $s\in I\mapsto v (s,\gamma(s)) $ is strongly measurable and satisfies
for all $a,b\in I, a<b$,
$$
\int_a^b \|v (s,\gamma(s))\| ds<+\infty.
$$
Equivalently, we define the space of locally absolutely continuous curves  $AC^1_{loc} (I;B)$ to be the space of all functions
$u: I\to B$ such that there exists $ m \in L^1_{loc}(I,dt)$ satisfying
\[\forall s,r \in I , \ s<r : \quad  ||u(s)-u(r)|| \leq \int_{s}^{r} m(t) dt\, .\]
Since separable dual Banach spaces satisfy the Radon-Nikodym property (see \cite{Ryan2002}), the  functions  in the space  $AC^1_{loc} (I;B)$ are continuous and almost everywhere differentiable on $I$ with a derivative $\dot{u}(\cdot)\in
L^1_{loc}(I,dt;B)$. Hence, a curve  $\gamma:I\to B$ is a mild solution of \eqref{chap3.ivpinh} if and only if $\gamma\in AC^1_{loc} (I;B)$, $\gamma(t_0)=x$  and for  almost all $t\in I$
\[\dot{\gamma}(t)=v(t,\gamma(t)).\]

\medskip\noindent
{\it Statistical Liouville  equation:}
When studying the statistical Liouville equation \eqref{chap4.intro.eq2}, the following notion will be useful.
\begin{definition}[Fundamental strongly total biorthogonal system]
\label{biorthogonal_system}
We say that the families $\{e_k\}_{k\in \NN}$ and $\{e^*_{k}\}_{k\in\NN}$ in $E, E^*$ respectively form a
\emph{fundamental strongly total biorthogonal system} if the following properties hold.
\begin{enumerate}[label=(\alph*)]
\item  \label{biorth.sys.1} ${\rm { Span}}\{e_k,k\in\NN\}$ is  dense in $E$ ({\it fundamental}),
\item \label{biorth.sys.2}  ${\rm { Span}}\{e^*_{k},k\in\NN\}$ is  dense in $B=E^*$ ({\it strongly total}),
\item \label{biorth.sys.3}  $\langle e^*_{k'}, e_k\rangle=\delta_{k',k}$, $\forall k,k'\in \NN$ ({\it biorthogonal}).
\end{enumerate}
\hfill$\square$
\end{definition}
We note that such an object exists in our framework.
\begin{lemma}
\label{biorthogonal_system_lemma}
Let $(B,\|\cdot\|)$ be a separable dual Banach space. Then a fundamental strongly total biorthogonal system $\{e_k\}_{k\in \NN}$, $\{e^*_{k}\}_{k\in\NN}$ as in Definition \ref{biorthogonal_system} exists.
\end{lemma}
 Lemma \ref{biorthogonal_system_lemma} is proved in \cite[Proposition 1.f.3]{MR0500056}.
We henceforth fix a system $\{e_k\}_{k\in \NN}$, $\{e^*_{k}\}_{k\in\NN}$ as in Definition \ref{biorthogonal_system}, whose existence is guaranteed by Lemma \ref{biorthogonal_system_lemma} above. This allows us to define a convenient class of cylindrical test functions.
\begin{definition}[Cylindrical test functions]
\label{cylindrical_test_functions}
A function $F: B=E^*\to \RR$ belongs to $\mathscr{C}_{c,cyl}^\infty(B)$ (resp.~$\mathscr{S}_{cyl}(B)$ or $\mathscr{C}^\infty_{b,cyl}(B)$)  if there exists $n\in\NN$ and $\varphi\in \mathscr{C}^\infty_{c}(\RR^{n})$ (resp.~$\mathscr{S}(\RR^n)$ or $\mathscr{C}^\infty_{b}(\RR^{n})$) such that
\begin{equation}
\label{cylindrical_function}
F(u)= \varphi(\langle u, e_1\rangle, \dots \langle u, e_n\rangle), \quad\forall u\in B=E^*.
\end{equation}
Here, $\mathscr{S}(\RR^n)$ denotes the Schwartz space and we have the inclusions
$$
\mathscr{C}^\infty_{c,cyl}(B) \subset \mathscr{S}_{cyl}(B) \subset \mathscr{C}^\infty_{b,cyl}(B)\,.
$$
\hfill$\square$
\end{definition}
\noindent Let us note that any $F \in \mathscr{C}^\infty_{b,cyl}(B)$ is  Fr\'echet differentiable with a differential $DF(u)\in  B^{*}=E^{**}$    identified  with an element of $E$. Hence, we denote simply $\nabla F$ for the differential of $F$ so that
$\nabla F(u)\in E$ for all $u\in B$. In particular,  we have
\begin{equation}
\label{Liouville_equation_solution_remark_1}
\nabla F(u)= \sum_{k=1}^{n} \partial_k \varphi(\langle u, e_1\rangle, \dots \langle u, e_n\rangle) \, e_k \in E\,.
\end{equation}

\medskip
Let $\mathscr{B}(B)$  and $\mathscr{P}(B)$  denote respectively  the Borel $\sigma$-algebra  and the space of  Borel probability measures on the Banach space  $(B,||\cdot||)$. We endow  $\mathscr{P}(B)$  with the narrow  topology.
\begin{definition}[Narrow continuity]
\label{narrow_continuity}
We say a curve  $(\mu_t)_{t\in I}$ in $\mathscr{P}(B)$ is \emph{narrowly continuous} if for any  bounded continuous  real-valued function $F \in \mathscr{C}_b(B, \RR)$, the map
\begin{equation}
\label{chap3.narrcont}
t\in I \mapsto \int_{B} F(u) \,\mu_t(du) 
\end{equation}
is continuous. \hfill$\square$
\end{definition}

 We say that a narrowly continuous curve $(\mu_t)_{t\in I}$ in $\mathscr{P}(B)$ satisfies the statistical Liouville   equation with respect to the Borel vector field $v:I\times B\to B$ if
\begin{equation}\label{chap3.le}
\frac{d}{dt} \int_{B} F(u) \, \mu_t(du)=  \int_{B} \langle v(t,u), \nabla F(u)  \rangle \,
\mu_t(du), \quad \forall F \in \mathscr{C}^\infty_{c,cyl}(B),
\end{equation}
in the sense of distributions on the interior of  $I$.    On the right hand side of \eqref{chap3.le}, the quantity $\langle v(t,u), \nabla F(u)  \rangle$ refers to the duality bracket $E^*,E$. In order   for the above statistical Liouville equation to make sense, one needs a further assumption on the vector field $v$ which ensures the integrability of the right hand side of \eqref{chap3.le}.

\begin{assumption}[Assumption on the vector field $v$]
\label{assumption_v}
We assume that $v: I \times B \rightarrow B$ is a Borel vector field such that
\begin{equation}\label{locinteg}
t\in I\mapsto \int_{B}\Vert v(t,u) \Vert  \,\mu_t(du) \in L^1_{loc}(I,dt).
 \end{equation}
\end{assumption}
In particular,  it follows that the duality pairing
$\langle v(t,u), \nabla F(u)  \rangle \equiv  \langle v(t,u), \nabla F(u)  \rangle_{E^{*},E}$ makes sense and satisfies
\begin{equation}
\label{Liouville_equation_solution_remark_2}
|\langle v(t,u), \nabla F(u)  \rangle| \leq \|v(t,u)\| \, \|\nabla F(u)\|_E \leq C \|v(t,u)\|\,,
\end{equation}
for some constant $C>0$ depending on $F$. Using \eqref{Liouville_equation_solution_remark_2} and Assumption \ref{assumption_v}, it follows that the right-hand side of \eqref{chap3.le} is finite for almost every $t \in I$.
 \begin{remark}\label{chap.3.rem}
According to Lemma \ref{equilocint}, the above condition  \eqref{locinteg} is equivalent to the existence of a non-decreasing positive function $\omega: \RR_+\to \RR^*_+$ such that
\begin{equation}
\label{chap3.s1.eq3}
\int_{I}\int_B  \|v(t,u)\| \,{\mu_t(du)} \,\frac{dt}{\omega(|t|)} <+\infty\,.
\end{equation}
The latter assumption is sometimes more convenient for our analysis.
 \end{remark}

\subsection{Main results}\label{subsec1.2}
We give in this subsection our main abstract results which hold true on any separable dual Banach space. For instance, one can consider initial value problems on finite-dimensional normed spaces or on separable reflexive Banach spaces like Lebesgue spaces $L^p$ for $p\in(1,+\infty)$.  Furthermore, one can also consider non reflexive Banach spaces like the sequence space $\ell^1(\NN)$  which is  a separable dual  space.  Our first result concerns the almost sure existence of global mild solutions.

\begin{theorem}[Global solutions]
\label{chap3.main.thm}
Let $B$ be a separable dual Banach space and $v: I\times B\to B$ a Borel vector field. Suppose that there exists $(\mu_t)_{t\in I}$ a narrowly continuous curve in $\mathscr{P}(B)$  such that \eqref{locinteg}  holds true and  assume  that  $(\mu_t)_{t\in I}$ satisfies the statistical Liouville equation \eqref{chap3.le}. Then there exists a universally measurable subset $\mathscr G$ of $B$ of total measure $\mu_{t_0}(\mathscr G)=1$ such that for any $x\in \mathscr G$ there exists a global mild solution to the initial value problem \eqref{chap3.ivpinh}.
\end{theorem}

\begin{remark}
The following comments are useful.
\begin{itemize}
\item The above theorem extends straightforwardly  to any Banach space that is isometrically isomorphic to a separable dual Banach space.
\item The assumption on  narrow continuity of $(\mu_t)_{t\in I}$ can be relaxed to weak narrow continuity given in Definition \ref{chap3.def.wnar}.
\item The  notions of mild solutions and the statistical Liouville equation are explained in Section  \ref{subsec1.1}.
\item A universally measurable set is a subset of a Polish space (here $B$) that is measurable with respect to every complete probability measure.  In particular,  $\mathscr G$ is $\mu_{t_0}$-measurable.
\item The above theorem   provides no information about uniqueness of mild solutions.
\end{itemize}
\end{remark}

Next, we introduce the notion of \emph{measurable flow}. In the sequel, it is convenient  to consider $I=\RR$.
\begin{definition}[Measurable flow]
\label{def.mes.flow}
Let  $B$ be a separable dual Banach space and $v: \RR\times B\to B$ a Borel vector field.
We say that the initial value problem \eqref{chap3.ivpinh} admits a measurable flow  $\phi_{t_0}^t$  with respect to a narrowly continuous curve $(\mu_t)_{t\in \RR}$ of probability measures in $\mathscr{P}(B)$ if
for all $t_0\in\RR$,
$$
\mathscr{G}_{t_0}=\{x\in B : \exists \,!\, \gamma_x \text{ a  global mild solution of \eqref{chap3.ivpinh} such that } \gamma_x(t_0)=x  \}
$$
are Borel sets of total measure $\mu_{t_0}(\mathscr{G}_{t_0})=1$ and the map $\phi_{t_0}^t$
\begin{eqnarray*}
\phi_{t_0}^t:  \mathscr{G}_{t_0}&\longrightarrow & \mathscr{G}_{t}\\
  x &\longmapsto & \gamma_x(t)
\end{eqnarray*}
is Borel measurable and satisfies for all $s,t,t_0\in\RR:$
\begin{itemize}
\item $\phi_{t_0}^{t_0}= \mathrm{Id}$;
\item $\phi_{s}^t\circ\phi_{t_0}^s= \phi_{t_0}^t$;
\item $t\in\RR\mapsto \phi_{t_0}^t(x)\in B$ continuous for any $x\in \mathscr{G}_{t_0}$;
\item $\mu_t=(\phi_{t_0}^t)_{\sharp}\mu_{t_0}$.
\end{itemize}
\hfill$\square$
\end{definition}
\begin{remark}
It is worth noticing that in ergodic theory there exists a similar notion of measurable flow (i.e.~a one parameter group of bijective measure-preserving transformations $T_t: X\to X$ on a measure space $(X,\mu)$ such that $(t,x)\in\RR\times X\mapsto T_t(x)\in X$ is measurable). Here, our Definition  \ref{def.mes.flow} is slightly different (see \cite{MR0097489} and Remark \ref{rem.mes.flow}).
\end{remark}
We give below two further results:
\begin{itemize}
\item [(i)] In Theorem \ref{thm_measurable_flow}, we deduce from Theorem \ref{chap3.main.thm} that if additionally we  have uniqueness for \eqref{chap3.ivpinh}
(i.e. there exists at most one global mild solution of \eqref{chap3.ivpinh} for each initial data $x\in B$ and $t_0\in\RR$), then  we can construct a \emph{measurable flow}  for the initial value problem \eqref{chap3.ivpinh}.
\item [(ii)]  Furthermore, in Theorem \ref{thm_flow_invariance}, we show that if the initial value problem \eqref{chap3.ivpinh} admits a \emph{measurable flow}  with respect to a narrowly continuous curve of probability measures $(\mu_t)_{t\in\RR}$, then $(\mu_t)_{t\in\RR}$ satisfies the statistical Liouville equation \eqref{chap3.le}.
\end{itemize}

\begin{theorem}[Construction of  a measurable flow]
\label{thm_measurable_flow}
Let $B$ be a separable dual Banach space and let $v: \RR\times B\to B$  be a Borel vector field. Let
$(\mu_t)_{t\in \RR}$ be a narrowly continuous curve in $\mathscr{P}(B)$  such that  \eqref{locinteg}  holds true and  such that $(\mu_t)_{t\in \RR}$ satisfies the statistical Liouville equation \eqref{chap3.le}. Assume that for an initial time $t_0\in\RR$ and any $x\in B$ the initial value problem  \eqref{chap3.ivpinh} admits at most one global mild solution. Then the initial value problem  \eqref{chap3.ivpinh} admits a measurable flow
as in  Definition \ref{def.mes.flow}.
\end{theorem}
\begin{remark}[Flow invariance]
In particular, according to  Theorem \ref{thm_measurable_flow} above, if  $(\mu_{t})_{t\in\RR}$ is a stationary solution of the statistical Liouville equation \eqref{chap3.le} (for all $t\in\RR$, $\mu_t=\mu_0$), then we have the flow invariance relation:
$$
(\phi_{0}^t)_{\sharp}\mu_{0}=\mu_0\,,\quad \forall t\in\RR.
$$
\end{remark}

\begin{theorem}[Liouville principle]
\label{thm_flow_invariance}
Let $B$ be a separable dual Banach space and $v: \RR\times B\to B$ a Borel vector field. Let $(\mu_t)_{t\in \RR}$  be a narrowly continuous curve in $\mathscr{P}(B)$  such that \eqref{locinteg}  holds true. Assume that the initial value problem \eqref{chap3.ivpinh} admits a measurable flow $(\phi_{t_0}^t)$ with respect to $(\mu_t)_{t\in \RR}$ as in Definition \ref{def.mes.flow}.
 Then the curve $(\mu_t)_{t\in \RR}$ satisfies the statistical Liouville equation \eqref{chap3.le}.
\end{theorem}
\begin{remark}[Stationary measure]
In particular, in   Theorem \ref{thm_flow_invariance} above, if  $(\mu_{t})_{t\in\RR}$ is  stationary
(i.e. for all $t\in\RR$, $\mu_t=\mu_0$), then $\mu_0$ is a stationary solution of the statistical Liouville equation \eqref{chap3.le}.
\end{remark}

\subsection{Application to ODEs}\label{chap3.application.ode}
Consider the  euclidean space $\RR^{2d}$ endowed with a symplectic structure given by a skew-symmetric matrix $J$ satisfying $J^2=-\rm I_{2d}$. Let $h: \RR^{2d}\to \RR$ be a  Borel   function in the  local Sobolev space
$W^{2,2}_{loc}(\RR^{2d}, \RR)$.  Furthermore, suppose that there exists a   non-negative $\mathscr{C}^1_b$-function $F:  \RR \to \RR$ such that
 \begin{equation}
 \label{chap3.app.ode.eq1}
0< \int_{\RR^{2d}}  \big(1+\Vert\nabla h(u)\Vert_{\RR^{2d}} \big) \,\, F(h(u)) \, L(du)<+\infty,
 \end{equation}
where $L$ denotes the Lebesgue measure over $\RR^{2d}$. Consider the initial value problem given by the
 Hamiltonian system:
 \begin{equation}
\label{chap3.ode}
\begin{cases}
\dot{\gamma}(t)=J \nabla h(\gamma(t)), \\
\gamma(0)=u_0 \in \RR^{2d}.
\end{cases}
\end{equation}

Observing  that   $\mu=\frac{F(h(\cdot)) L}{\int F(h(u)) L(du)} \in \mathscr{P}(\RR^{2d})$  is a stationary solution to  the  statistical Liouville equation   \eqref{chap3.le}
with the vector field $v=J \nabla h$,  yields  the following consequence of  Theorem \ref{chap3.main.thm}.
\begin{corollary}
\label{cor.ode}
Assume that $h\in W^{2,2}_{loc}(\RR^{2d}, \RR)$ and the condition  \eqref{chap3.app.ode.eq1} is satisfied. Then the Hamiltonian system \eqref{chap3.ode} admits a global mild solution for Lebesgue-almost any initial condition $u_0\in \RR^{2d}$.
\end{corollary}

As an illustrative example for $d\geq 5$, one can take $F(t)=e^{-\beta t}, t\in\RR_+$, for some $\beta>0$ and
$$
h(u)= \phi(p)+\Vert q\Vert^2_{\RR^d}+\frac{1}{\Vert q\Vert^\alpha_{\RR^d}}\,,
$$
with $\phi\in W^{2,2}_{loc}(\RR^{d}, \RR_+)$ such that {$F(\phi)$ and $\Vert\nabla\phi\Vert_{\RR^d} F(\phi)$ belong to $L^1(\RR^d)$} and {$\alpha < \frac d 2-2$} (here $u=(q,p)\in\RR^d\times\RR^d$ and $\Vert \cdot\Vert_{\RR^d}$ is the euclidean norm).

\medskip
\paragraph{\textbf{Counter-example:}}
We recall a counter-example from the work of Cruzeiro \cite{MR724704}, which shows that Assumption \ref{assumption_v} cannot be omitted. Indeed, consider the time-independent $\mathscr{C}^\infty$-vector field $v: \RR\times \RR^2\to \RR^2$  given by
\begin{equation}\label{vect.cont-examp}
v(t,u)= \big( q^2, (2q-q^3) e^{p^2/2} \int_{p}^{+\infty} e^{-t^2/2} dt \big), \qquad u=(q,p)\in\RR^2,
\end{equation}
and  the stationary family of probability  measures $\mu_t=\mu_0$ where $\mu_0$ is the standard centered Gaussian measure on $\RR^2$. Then, $(\mu_t)_{t\in\RR}$  satisfies the statistical Liouville equation \eqref{chap3.le} with the above vector field $v$.  Moreover, the initial value problem \eqref{chap3.ivpinh} with $v$ as in \eqref{vect.cont-examp} leads to the ODE
$$
\dot{q}(t)=q(t)^2\,,
$$
which has non global (unique) solutions for any initial condition $q(0)\neq 0$.   On the other hand,  one checks that  $v$ does not satisfy Assumption  \ref{assumption_v}.  This shows the existence of a $\mathscr{C}^\infty$-vector field and a stationary probability measure solving the statistical Liouville equation, but for which the conclusion of  Theorem \ref{chap3.main.thm} does not hold because Assumption  \ref{assumption_v} is not satisfied.  In this respect, one can interpret the integral condition \eqref{locinteg} or \eqref{chap3.s1.eq3} as an almost sure non-blow up assumption.

\subsection{Application to  PDEs}\label{chap3.application}
Consider a complex Hilbert space $(H, \|\cdot\|_{H})$ and a self-adjoint operator $A: D(A)\subset H \to H$ such that  there exists a constant $c>0$,
$$
A\geq c \, \mathds 1,
$$
and $A$ having a compact resolvent. So, there exist a sequence of eigenvalues $\{\lambda_k\}_{k\in\NN}$ and an O.N.B of eigenvectors
\begin{equation}
\label{e_k_definition}
Ae_k=\lambda_k e_k
\end{equation}
such that $Ae_k=\lambda_k e_k$ for all $k\in\NN$. Furthermore, assume that there exists $s\geq 0$ such that
\begin{equation}
\label{chap.3.sob.s}
\sum_{k\in\NN} \lambda_k^{-(s+1)} <+\infty.
\end{equation}
Then, one can define Sobolev spaces with positive exponent $r>0$ as
$$
H^{r}= (D(A^{r/2}), \|A^{r/2} \cdot\|_{H})
$$
and Sobolev spaces with negative exponent as
$$
H^{-r}= \overline{(H, \|A^{-r/2}\cdot\|_{H})}^{\rm \, completion} .
$$
From now on, we regard $H, H^r, H^{-r}$ as real Hilbert spaces endowed respectively with the scalar products
$$
\langle \cdot , \cdot \rangle_{H_{\RR}}={\rm Re}\langle \cdot , \cdot \rangle_{H}, \quad \langle \cdot , \cdot \rangle_{H^r_{\RR}}={\rm Re}\langle \cdot , A^r \cdot \rangle_{H}, \quad \langle \cdot , \cdot \rangle_{H^{-r}_{\RR}}={\rm Re}\langle \cdot , A^{-r} \cdot\rangle_{H},
$$
and denote them respectively by $H_{\RR}, H_{\RR}^r, H_{\RR}^{-r}$ (note that  ${\rm Re (\cdot)}$ refer to the real part). Then it is well-known that there exists a unique centred Gaussian probability measure $\nu_0$ on the Sobolev space $H^{-s}$ with $s\geq 0$ satisfying \eqref{chap.3.sob.s} and such that for all
$\xi\in H^{-s}$,
\begin{equation}
\label{chap3.fgaussian}
\int_{H^{-s}} \exp({i \langle u, \xi\rangle_{H_{\RR}^{-s}}}) \; \nu_0(du)=  \exp({-\frac 1 2 \langle \xi, A^{-(s+1)}\xi\rangle_{H_{\RR}^{-s}}}) \,.
\end{equation}
Once we have such a centred  Gaussian measure $\nu_0$, one can define the Gross-Sobolev space
$$
\mathbb{D}^{1,2}(\nu_0)=\{F\in L^2(H^{-s},\nu_0) : \nabla F \in L^2(H^{-s},\nu_0; H^{-s}) \},
$$
where here $\nabla F$ is the Malliavin derivative of $F$ (see for instance \cite{MR2200233} or \cite{MR4571599} for brief details). In particular, $\mathbb{D}^{1,2}(\nu_0)$ is a Hilbert space when endowed with the inner product
$$
\langle F, G\rangle_{\mathbb{D}^{1,2}(\nu_0)}= \langle F, G\rangle_{L^{2}( H^{-s},\nu_0)}+
\langle \nabla F, \nabla G\rangle_{L^{2}( H^{-s},\nu_0; H^{-s})}.
$$
Our purpose is to prove that the initial value problem \eqref{chap3.ivpinh} admit  global solutions
for $\nu_0$-almost any initial condition $x\in H^{-s}$ (here $B=H^{-s}$) when the vector field $v:\RR\times H^{-s} \to H^{-s}$ is given explicitly by
\begin{equation}
\label{chap3.vdisp}
v(t,u)= -i e^{it A} \nabla h_{NL}( e^{-it A} u),
\end{equation}
and $h_{NL}$ is any function in the Gross-Sobolev space $\mathbb{D}^{1,2}(\nu_0)$ satisfying the condition
\begin{equation}
\label{chap3.cdh}
\exp(-h_{NL})\in L^2(H^{-s},\nu_0).
\end{equation}
Thanks to the above assumptions, the following Gibbs measure is well-defined.
\begin{equation}
\label{chap3.gibbs}
d\mu_0= \frac{e^{-h_{NL}} d\nu_0}{\int_{H^{-s}} e^{-h_{NL}}  d\nu_0}\,.
\end{equation}
Then within the above framework, we prove in Section \ref{chap3.sec.app.proof} the following result.

\begin{proposition}
\label{chap3.invdisp}
Consider the time-dependent push-forward  Gibbs measures
\begin{equation}
\label{chap3.gibbs_t}
\mu_t= (e^{itA})_\sharp \mu_0.
\end{equation}
Then for all $t\in\RR$ and any $F\in \mathscr{C}_{b,cyl}^\infty(H^{-s})$,
\begin{equation}
\label{chap3.eq.livh-s}
\frac{d}{dt}  \int_{H^{-s}}  F(u)  \, \mu_t(du) = \int_{H^{-s}} \langle v(t,u), \nabla F(u)\rangle \, \mu_t(du).
\end{equation}
\end{proposition}
This implies that there exists a narrowly continuous curve $(\mu_t)_{t\in\RR}$ in $\mathscr{P}(H^{-s})$, given in \eqref{chap3.gibbs_t}, satisfying the statistical Liouville equation \eqref{chap3.eq.livh-s} and the integrability condition
\begin{equation*}
\int_{\RR}\int_{H^{-s}}  \|v(t,u)\|_{H^{-s}} \, \mu_t(du) \,\frac{dt}{\langle t\rangle ^2} = \int_{\RR}\frac{dt}{\langle t\rangle ^2} \; \int_{H^{-s}}  \|\nabla h_{NL}(u)\|_{H^{-s}} \, \mu_0(du) <+\infty,
\end{equation*}
since $ e^{-h_{NL}}\in L^2( H^{-s},\nu_0)$ and $\nabla h_{NL} \in L^2( H^{-s},\nu_0; H^{-s})$  (recall that $h_{NL}\in\mathbb{D}^{1,2}(\nu_0)$),  which corresponds to Assumption \ref{assumption_v} (equivalently  to \eqref{chap3.s1.eq3} with $\omega(t)=\langle t \rangle^2$). Thus, Theorem \ref{chap3.main.thm} yields the following statement.

\begin{corollary}
\label{chap3.cor.res}
For any nonlinear functional $h_{NL}: H^{-s}\to \RR$ belonging to the Gross-Sobolev space $\mathbb{D}^{1,2}(\nu_0)$ and satisfying
$$
\exp(-h_{NL})\in L^2(H^{-s},\nu_0),
$$
the initial value problem
\begin{equation}
\label{chap3.ivpH-s}
\begin{cases}
i \partial_t u(t)=Au(t)+\nabla h_{NL}(u(t)), \\
u(0)=x \in H^{-s},
\end{cases}
\end{equation}
admit global mild solutions for $\nu_0$-almost any initial data in $H^{-s}$.
\end{corollary}
Such a result is new to the best of our knowledge. It is a straightforward  consequence of the  Theorem \ref{chap3.main.thm}  and Proposition \ref{chap3.invdisp}.  The main point in Corollary \ref{chap3.cor.res} is that the existence of well-defined Gibbs measure provides a global solution to the statistical Liouville equation and hence by Theorem \ref{chap3.main.thm}, one deduces the almost sure existence of global solutions.   It is clear therefore  that there is a tight connection between Theorem  \ref{chap3.main.thm} and the subject of Gibbs measures and low regularity solutions of dispersive PDEs. In fact,  several examples of concrete PDEs like NLS, Hartree and Wave equations can be recast as the above initial value problem \eqref{chap3.ivpH-s}. One needs only to specify the Hilbert space $H$, the operator $A$ and the nonlinear functional $h_{NL}$.

\bigskip
Gibbs measures for nonlinear dispersive equations are well-studied and the literature on the subject is quite large as was summarised above. To  highlight the connection of our results with this topic, we provide here some applications of Corollary \ref{chap3.cor.res}  to concrete examples. Our aim is not to give all the possible applications, but rather to illustrate our method.

\medskip
\paragraph{\textbf{Hartree and NLS equations:} }
Let $H$ be  the Hilbert space $L^2(\TT^d)$ with $\TT^d=\RR^d/(2\pi\mathbb Z^d)$  the flat $d$-dimensional torus. Take the operator $A$ as
\begin{equation}
\label{chap3.eq.38}
 A=-\Delta+ 1\,,
\end{equation}
where $\Delta$ is the Laplacian  on $\TT^d$. So, the family $\{e_k=e^{ik x}\}_{k\in\mathbb Z^d}$ forms an O.N.B of eigenvectors for the operator $A$ which admits a compact resolvent. Now, consider an exponent $s\geq 0$ such that
\begin{equation}
\label{chap3.s_choice}
\fbox{$s>\frac{d}{2}-1$}\,,
\end{equation}
and define the Sobolev space $H^{-s}$ accordingly. Denote by $\nu_0$  the  well-defined centred Gaussian measure on  $H^{-s}$ given by \eqref{chap3.fgaussian}. Then, we list  some nonlinear functionals, $h_{NL}:H^{-s}\to \RR$, for which Corollary \ref{chap3.cor.res} applies; specified according to the dimension $d$ and the type of equation.

\begin{itemize}
\item \emph{The Hartree equation on $\TT$}: Let $V:\TT\to\RR$ be  a nonnegative even $L^1$ function and
\begin{equation}
\label{chap3.eq.nl.3}
h_{NL}(u)=\frac{1}{4}\;\int_{\TT}\,\int_{\TT} |u(x)|^2 \,V(x-y)\, |u(y)|^2 \,dx\,dy \,.
\end{equation}
\item \emph{The Hartree equation on $\TT^d$, $d=2,3$}: Let  $V \in L^1(\mathbb{T}^d)$ be
even and of positive type such that there exist $\epsilon>0$ and $C>0$ with the property that for all $k \in \mathbb{Z}^d$,
\begin{equation}
\label{V_hat_estimates}
\begin{cases}
\hat{V}(k) \leq \frac{C}{\langle k \rangle^{\epsilon}} &\mbox{if } d=2
\\
\hat{V}(k) \leq \frac{C}{\langle k \rangle^{2+\epsilon}} &\mbox{if } d=3\,.
\end{cases}
\end{equation}
Then take
\begin{equation}
\label{chap3.eq.nl.3_Wick}
h_{NL}(u) = \frac{1}{4}\;\int_{\TT^{d}}\,\int_{\TT^{d}} :|u(x)|^2: \,V(x-y)\, :|u(y)|^2: \,dx\,dy\,.
\end{equation}
\item \emph{The NLS equation on $\TT$}: Let $r \in \NN$ and let
\begin{equation}
\label{chap3.eq.nl.4}
h_{NL}(u)=\frac{1}{2r}\;\int_\TT |u(x)|^{2r}\, dx \geq 0\,.
\end{equation}
\item \emph{The NLS equation on $\TT^2$:} Let $r \in \NN$ with
\begin{equation}
\label{chap3.eq.nl.1}
h_{NL}(u)= \frac{1}{2r} \int_{\TT^{2}} :|u|^{2r}: \,dx\,.
\end{equation}
\end{itemize}
Here, the notation $: \;:$ refers to Wick ordering with respect to the Gaussian measure $\nu_0$. See for instance \cite{MR3844655} for a self-contained construction of these Wick ordered nonlinearities \eqref{chap3.eq.nl.3_Wick}-\eqref{chap3.eq.nl.1}.

\medskip
\paragraph{\textbf{Wave equations:}}  Consider the  Hilbert space $H=L_\RR^2(\TT^d)\oplus L_\RR^2(\TT^d)$  where  $L_\RR^2$ is  the space of real-valued square integrable  functions.  For  $s\in\RR$ satisfying \eqref{chap3.s_choice},  define  the  Sobolev space
\begin{equation}
\label{chap3.sob.wav}
H^{-s}=H^{-s}_\RR(\TT^d)\oplus H^{-s-1}_\RR(\TT^d)\,.
\end{equation}
The nonlinear wave equation takes the form
\[
\partial_t \begin{pmatrix} u \\ v\end{pmatrix} = \begin{bmatrix} 0 & \mathds 1\\ \Delta-\mathds 1 & 0\end{bmatrix}
\begin{pmatrix} u \\ v\end{pmatrix}+ \begin{pmatrix} 0 \\ -\nabla h_{NL}(u), \end{pmatrix}
\]
with $h_{NL}$ a nonlinear functional on $H^{-s}_\RR(\TT^d)$ given by
\begin{equation}
\label{nonlin_wa}
h_{NL}=
\begin{cases}
  \eqref{chap3.eq.nl.3} \text{ or } \eqref{chap3.eq.nl.4} & \text{ if } d=1,  \\
  \eqref{chap3.eq.nl.3_Wick} \text{ or } \eqref{chap3.eq.nl.1} & \text{ if } d=2,\\
  \eqref{chap3.eq.nl.3_Wick} & \text{ if } d=3. \\
\end{cases}
\end{equation}
The Gaussian measure $\nu_{0}$ in the case of the wave equation is defined as the product measure
$$
\nu_{0}=\nu_{0}^1\otimes \nu_{0}^2,
$$
where   $\nu_{0}^1$ and  $\nu_{0}^2$ are  Gaussian measures on the distribution space $\mathscr{D}'(\TT^d)$ with covariance operators $ (-\Delta+\mathds 1)^{-1}$ and $\mathds 1$ respectively.   Moreover,  one can define  rigorously  the Gibbs measure for the nonlinear wave equation as
$$
d\mu_0=\frac{e^{- h_{NL}} \,d\nu_{0}^1\otimes \nu_{0}^2}{\int e^{-h_{NL}} d\nu_0^1}.
$$
We recall the following result, proved in \cite{MR4571599}.
\begin{proposition}
\label{chap3.prop.GS1}
The nonlinear Borel functionals $h_{NL}(\cdot)$ given by \eqref{chap3.eq.nl.3}, \eqref{chap3.eq.nl.3_Wick}, \eqref{chap3.eq.nl.4} and \eqref{chap3.eq.nl.1} belong to the Gross-Sobolev space $\mathbb D^{1,2}(\nu_{0})$.
\end{proposition}
Hence, we recover the following statements as a consequence of Corollary \ref{chap3.cor.res}.
\begin{corollary}
\label{chap3.app.nl}
The above Hartree, NLS and Wave equations  with the nonlinearities   \eqref{chap3.eq.nl.3}, \eqref{chap3.eq.nl.3_Wick}, \eqref{chap3.eq.nl.4} and \eqref{chap3.eq.nl.1}  admit global mild solutions almost surely with respect to the Gaussian measure $\nu_0$.
\end{corollary}

Before proceeding with the proof, we note a few applications of our results.

\medskip
\paragraph{\textbf{Nonlinear (an)-harmonic equations:}}
One can consider nonlinear Schr\"odinger type equations with (an)-harmonic oscillators on $\RR^d$, i.e.\ with
\begin{equation}\label{eq.A.anharmonic}
A= -\Delta+|x|^\alpha+1\,,
\end{equation}
for $\alpha>0$.  In this case, the Hilbert space we consider is $H=L^2(\RR^d)$. The Gaussian measure $\nu_0$ on $H^{-s}_{\RR}$ is defined as in \eqref{chap3.fgaussian} with $A$ as in
\eqref{eq.A.anharmonic}. The probability measure $\nu_0$ is well-defined provided that the assumption \eqref{chap.3.sob.s} on the eigenvalues of $A$ is satisfied.
Let us now verify the range of $s$ for which the latter is true.
We first recall the Lieb-Thirring inequality \cite{MR2253013}, which states that for $\gamma>\frac{d}{2}$, we have
\begin{equation} \label{Lieb-Thirring_inequality}
\sum_{k \in \NN} \lambda_k^{-\gamma} \leq C(\gamma) \int_{\RR^d} (|x|^{\alpha}+1)^{\frac{d}{2}-\gamma}\,dx\,.
\end{equation}
By \eqref{Lieb-Thirring_inequality}, we deduce that for
\begin{equation}
\label{eq.an.harmonic.assumpt}
   s> \frac{d}{2}+\frac{d}{\alpha}-1\,,
\end{equation}
the assumption  \eqref{chap.3.sob.s} is satisfied. Hence, we have all the ingredients to apply Corollary
\ref{chap3.cor.res}  and obtain the following statement.
\begin{corollary}
\label{chap3.app.anharmo.gwp}
Assume \eqref{eq.an.harmonic.assumpt} and consider any nonlinear functional $h_{NL}: H^{-s}\to \RR$ belonging to the Gross-Sobolev space $\mathbb{D}^{1,2}(\nu_0)$ such that
$$
\exp(-h_{NL})\in L^2(H^{-s},\nu_0).
$$
Then the nonlinear (an)-harmonic equation on $\RR^d$,
$$
i\partial_t u=-\Delta u+ |x|^\alpha u+u+\nabla h_{NL}(u)\,,
$$
admits global mild solutions on $H^{-s}$  almost surely with respect to the Gaussian measure $\nu_{0}$ in \eqref{chap3.fgaussian}, with $A$ as in \eqref{eq.A.anharmonic}.
\end{corollary}
Note that such a result still holds true under the perturbation of the (an)-harmonic oscillator \eqref{eq.A.anharmonic}  by a potential (see for instance the spectral asymptotics in \cite{MR4554532} and the references therein).

\medskip
\paragraph{\textbf{Nonlinear dispersive equations on bounded domains or manifolds:}}
Instead of working on the torus $\TT^d$ or on the whole space $\RR^d$, it is  possible to consider the  Hartree, NLS and Wave equations on bounded domains or on compact Riemannian  manifolds without boundary. In particular, Corollary \ref{chap3.cor.res} holds true in the following two cases.
\begin{itemize}
\item Let $\Omega$ be a bounded open domain in $\RR^d$. Take the Hilbert space $H=L^2(\Omega)$ and the operator
\begin{equation}
A=-\Delta_{\Omega}+ c\,,
\end{equation}
with Dirichlet or Neumann  boundary conditions. Here, $\Delta_{\Omega}$ is the Laplace operator on $\Omega$ and $c>0$ is a constant chosen such that $A$ is positive.  Thanks to Weyl's law (see \cite[Chapter 14]{MR2952218}), the assumption  \eqref{chap.3.sob.s} on the eigenvalues of the Laplacian $-\Delta_{\Omega}$ is satisfied whenever
 $$
 s>\frac{d}{2}-1\,.
 $$
 \item Let $(\mathcal M, g)$ be a $d$-dimensional compact Riemannian manifold without boundary.
Take the Hilbert space $H=L^2(\mathcal M)$ and consider  $A$ as the  Laplace-Beltrami operator
 \begin{equation}
A=-\Delta_{g}\,.
\end{equation}
Thanks to Weyl's law,  the assumption  \eqref{chap.3.sob.s}  on the eigenvalues of the  Laplace-Beltrami operator $-\Delta_g$  is satisfied whenever
 $$
 s>\frac{d}{2}-1\,.
 $$
\end{itemize}
As an example of a nonlinear functional $h_{NL}$ in this framework, one can consider the Wick-ordered nonlinearity \eqref{chap3.eq.nl.1} on bounded domains in $\RR^2$ or on $2$-dimensional compact Riemannian manifolds without boundary. We refer the reader to the discussions and full details given in \cite[Proposition 4.3, 4.5 and 4.6]{MR3844655} and \cite[Section 1.2]{MR4133695} for the rigorous construction of this type of functionals.

\medskip
\paragraph{\textbf{Fluid mechanics equations:}}

In this paragraph,  we follow the work of \cite{MR3818404} and refer the reader to the references therein for more details on the Euler and modified  SQG equations. Indeed,  the mSQG equation takes the form:
\begin{equation}
\label{chap3.msqg}
\begin{cases}
\partial_t \theta=-\big( \varphi \cdot\nabla \big)\theta , \qquad x\in\TT^2,  \; t>0, \\
\varphi = R^\perp |D|^{-\delta}\theta,
\end{cases}
\end{equation}
for $\delta>0$.  Here $\theta:\TT^2\to \RR, \varphi:\TT^2\to \RR^2$ are functions, $|D| := (-\Delta)^{1/2}$ and $R^\perp := \nabla^\perp |D|^{-1}$ denotes the Riesz transform with $\nabla^\perp =
(-\partial_{x_2}, \partial_{x_1})$.
The case $\delta =1$ in the above  equation corresponds  to the 2D Euler equation.   The streamline formulation ($u=|D|^{-1}\theta$) of the above mSQG equation yields
\begin{equation}
\label{chap3.msqg.bis}
\begin{cases}
\partial_t u=-|D|^{-1} \big( \varphi \cdot\nabla \big)|D| u , \\
\varphi = \nabla^\perp |D|^{-\delta} u,
\end{cases}
\end{equation}
so that the original mSQG equation is rephrased as an initial value problem with an autonomous  vector field  given by
\begin{equation}\label{chap.3.vectmsqg}
v(u)\equiv v(t, u)= -|D|^{-1} \big( \nabla^\perp |D|^{-\delta} u \cdot\nabla \big)|D|u\,.
\end{equation}
Consider now the centred Gaussian measure $\nu_0$ defined on  the negative the Sobolev space $H^{-s}$, $s>0$, as in the previous section (with $A=-\Delta$ and the Hilbert space $L_0^2(\TT^2)$ of mean zero square integrable functions). There are nice results on one hand  by Albeverio and Cruzeiro \cite{MR1051499} for the Euler equation and on the other by  Nahmod--Pavlovi\'{c}--Staffilani--Totz \cite{MR3818404} for the mSQG equation ($0<\delta\leq 1$), establishing the existence of solutions  for arbitrarily large lifespan and for almost initial data in the spaces $H^s$, $s<-2$, with respect to the Gaussian measure $\nu_0$.   Actually, thanks to preliminary results in \cite{MR3818404}, one can  apply our  Theorem \ref{chap3.main.thm}  to these equations too. Indeed, take $\mu_t=\nu_0$, $B=H^{-s}, s>2,$ and the vector field $v$ as in \eqref{chap.3.vectmsqg}, then in \cite[Proposition 4.1]{MR3818404} it is proved that $v \in  L^2(H^{-s},\nu_0; H^{-s})$ for all $s >2$.  This implies that the integrability condition \eqref{chap3.s1.eq3} is satisfied with any $\omega$ such that $\omega^{-1}\in L^1(\RR_+,dt)$. Moreover, the stationary  Liouville  equation  \eqref{chap3.le} is satisfied by the Gaussian measure $\nu_0$ thanks to  the  proof of  \cite[Lemma 5.1]{MR3818404}. Hence, as a consequence of Theorem   \ref{chap3.main.thm}, we obtain  $\nu_0$-almost surely the existence of global solutions to the mSQG equation \eqref{chap3.msqg}-\eqref{chap3.msqg.bis} in $\mathscr{C}(\RR_+, H^{-s})$ for all $s>2$.
In particular, our application of Theorem \ref{chap3.main.thm} yields an improvement of the result \cite{MR3818404} as it gives  almost sure global solutions  instead of arbitrarily large lifespan solutions.

\section{Global superposition principle}
\label{chap.3.sec.2}

Our purpose in this part is to state and prove the global superposition principle (Proposition \ref{propinfinite}). For that, we need to introduce the path space
\begin{equation}\label{mathcalH}
 {\mathfrak{X}}:=B\times \mathscr{C}(I; B),
 \end{equation}
 composed of  pairs $(x,\gamma)$ where $x\in B$ and $\gamma$ is a continuous curve in $(B,\Vert\cdot\Vert)$.
Denote, for  any $\gamma\in\mathscr{C}(I; B)$ and $m\in\NN$,
 \begin{equation}
 \label{chap.3.norm-m}
 \Vert \gamma\Vert_{m} = \sup_{[-m,m]\cap I} \Vert \gamma(t)\Vert ,
 \end{equation}
with the convention $ \Vert \gamma\Vert_{m}=0$ if $[-m,m]\cap I$ is the empty set. Since the interval $I$ is unbounded and closed, it is convenient to equip the space $\mathscr{C}(I; B)$ with the compact-open topology which is metrizable in our case with the metric:
\begin{equation}
\label{chap.3.compact-open}
d_0(\gamma_1,\gamma_2)=  \sum_{m\in\NN} \frac{1}{2^m} \frac{\Vert \gamma_1-\gamma_2\Vert_{m}}{1+\Vert \gamma_1-\gamma_2\Vert_{m}}\,.
\end{equation}
Then, we  accordingly define a distance $d$ over  the product space  $\mathfrak{X}=B\times \mathscr{C}(I; B)$,
\begin{equation}
 \label{d_metric}
 d\big((x_1,\gamma_1); (x_2,\gamma_2)\big)=\Vert x_1-x_2\Vert+d_0(\gamma_1,\gamma_2).
 \end{equation}
Furthermore, we define, for each $t \in I$, the  evaluation map
 \[ \begin{array}{lrcl}
\Xi_t \, : & \mathfrak{X} & \longrightarrow & B \\
&(x,\gamma) & \longmapsto & \displaystyle \gamma(t).
\end{array} \]

\noindent
Now, we are in position to state the global superposition principle. Recall that $I$ is an unbounded closed interval and $B$ is a separable dual Banach space equipped with a biorthogonal system satisfying  \ref{biorth.sys.1}-\ref{biorth.sys.3}  given in Definition \ref{biorthogonal_system}.
\begin{proposition}[Global superposition principle]\label{propinfinite}
Let $v:I\times B \rightarrow B$ be a Borel vector field. Let   $(\mu_t)_{t\in I}$ be a  narrowly continuous curve in $\mathscr{P}(B)$ satisfying  the integrability condition \eqref{locinteg}  and the statistical  Liouville equation   \eqref{chap3.le}.   Then, there exists  a Borel probability measure $\eta \in \mathscr{P}(\mathfrak X)$ such that:
\begin{enumerate}[label=(\roman*)]
\item \label{chap.3.(i)} $\eta$  concentrates on the set of pairs $(x,\gamma)$ such that $\gamma  \in  AC^1_{\it loc}(I;B)$ is a mild solution of the initial value problem \eqref{chap3.ivpinh} for a.e. $t\in I$ with $\gamma(t_0)=x$.
\vskip1mm
\item \label{chap.3.(ii)} $\mu_t $ is equal to the image measure $(\Xi_t)_{{\sharp}} \eta$, for all $t\in I$: {\it i.e.,}  $\forall \mathcal O \in \mathscr{B}(B), \; \mu_t(\mathcal O)=  \eta\big( \Xi_t^{-1}(\mathcal O)\big)$.
\end{enumerate}
\end{proposition}

\medskip
The remaining part of this section is dedicated to the proof the above proposition.  In subsection \ref{subsec.weak-star}, we introduce  convenient weaker topologies on the space $B$ and the path space $\mathfrak{X}$.  In Subsection \ref{project}, we set up a finite-dimensional projection argument. In Subsection \ref{finitedimension}, we prove Proposition \ref{propinfinite} when $B=\RR^d$.  Then, we extend such a result to infinite-dimensional separable dual Banach spaces in Subsection \ref{proofofprop2}.

\subsection{Weak topologies}
\label{subsec.weak-star}
The following  topological and measure theoretical aspects will be very useful in the proofs of our main Theorem \ref{chap3.main.thm} and the global superposition principle (Proposition \ref{propinfinite}).

 It is useful  to introduce a  norm over $B$ that ensures relative compactness of bounded sets,
\begin{equation}
\label{chap3.norm.def}
\|x\|_{*}:=\sum_{k\in\NN} \frac{1}{2^k \|e_k\|_{E}} \,\vert \langle x,e_k \rangle \vert, \qquad x\in B.
\end{equation}
Here   $\langle \cdot,\cdot \rangle$ is the duality  pairing of  $E^*$, $E$ and  $\lbrace e_k \rbrace_{k\in\NN} $,
$\lbrace e^*_k \rbrace_{k\in\NN} $ is the fixed biorthogonal system in Definition  \ref{biorthogonal_system}.  Actually,  the norm $\|\cdot\|_{*}$ yields a distance on $B=E^*$ which metrizes the weak-* topology $\sigma(E^*,E)$ on bounded sets.
For convenience, we denote by $B_w$ the space $B$ endowed with the above norm \eqref{chap3.norm.def} and remark  that $B_w=(B,\|\cdot\|_{*})$ is  separable.

\medskip
Recall that  $\mathscr{P}(B)$ denotes the space of Borel probability measures on $(B,\Vert\cdot\Vert)$. The following lemma shows that
$\mathscr{P}(B)$ is unchanged as a set, if we equip  the space $B$ with the norms $\Vert\cdot\Vert$ or  $\Vert\cdot\Vert_{*}$.
\begin{lemma}
The $\sigma$-algebras of  Borel sets of $(B,\|\cdot\|)$ and $(B, \|\cdot\|_{*})$ coincide.
\end{lemma}
\begin{proof}
See \cite[Lemma C1]{alc2020}.
\end{proof}
It is useful  to distinguish two narrow topologies  over  $\mathscr{P}(B)$.  Namely, the (strong) narrow topology when $B$ is equipped with the original norm $\Vert\cdot\Vert$ and the "weak" narrow topology  when $B$ is endowed with the norm $\Vert\cdot\Vert_{*}$.
\begin{definition}[Weak narrow topology]
\label{chap3.def.wnar}
We say that a sequence $(\mu_n)_{n\in \NN}$ of Borel probability measures in $\mathscr{P}(B)$ converges \emph{weakly narrowly} to $\mu\in \mathscr{P}(B)$ if for every bounded continuous function $F \in \mathscr{C}_b(B_w;\RR)$,
\[ \lim_n  \int_{B} F(u) \,\mu_n(du)=\int_{B} F(u) \,\mu(du).\]
In such a case, we denote
$$
{\mu}_n \underset{n\rightarrow\infty}{\rightharpoonup} \mu\,.
$$
Accordingly, a curve $(\mu_t)_{t\in I}$ in $\mathscr{P}(B)$ is said to be  \emph{weakly narrowly continuous} if the real-valued map
\[ t \in I \longmapsto \int_{B} F(u) \,\mu_t(du),\]
is continuous for every  $F \in \mathscr{C}_b(B_w;\RR)$.
\hfill$\square$
\end{definition}
\begin{remark}
Note  that in finite dimensions the weak narrow and (strong) narrow topologies  coincide.
\end{remark}

\medskip
On the other hand, we  define  similarly  new distances on $\mathscr{C}(I; B)$ and the path space $\mathfrak{X}=B\times \mathscr{C}(I; B)$ given respectively by
 \begin{equation}
 \label{chap.3.d0star}
 d_{0,*}\big(\gamma_1; \gamma_2\big)=\sum_{m\in\NN} \frac{1}{2^m} \frac{\Vert \gamma_1-\gamma_2\Vert_{*,m}}{1+\Vert \gamma_1-\gamma_2\Vert_{*,m}},
 \end{equation}
 and
 \begin{equation}
 \label{dstar_metric}
 d_*\big((x_1,\gamma_1); (x_2,\gamma_2)\big)=\Vert x_1-x_2\Vert_{*}+ d_{0,*}\big(\gamma_1; \gamma_2\big),
\end{equation}
where
 \begin{equation}
 \label{T_norm}
 \Vert \gamma\Vert_{*,m} = \sup_{[-m,m]\cap I} \Vert \gamma(t)\Vert_{*}\,,
 \end{equation}
 with the convention $ \Vert \gamma\Vert_{*,m}=0$ if $[-m,m]\cap I$ is the empty set.
It is clear that the topology induced by   $d_*$ is coarser than the one induced by the distance $d$ given in \eqref{d_metric}.  Moreover, $d_*$ is the induced metric on $\mathfrak{X}$ corresponding  to  the  product topology between $(B,\Vert\cdot\Vert_*)$ and the space $\mathscr{C}(I; B_w)$ endowed with the compact-open topology.

Recall that a Polish space is a Hausdorff topological space homeomorphic to a separable complete metric space while a Suslin space is
a Hausdorff topological space which is the image of a Polish space under a continuous map.  { In particular, in our case, $(B,\Vert\cdot\Vert_*)$ is a Suslin space.}
\begin{lemma}
\label{sec.glob.lem.1}
With the distances  $d$ and $d_*$ given respectively in  \eqref{d_metric} and \eqref{dstar_metric}, we have:
\begin{itemize}
\item $(\mathscr{C}(I; B), d_0)$ and $(\mathfrak{X}, d)$    are  Polish spaces.
\item  $(\mathscr{C}(I; B), d_{0,*})$ and $(\mathfrak{X}, d_*)$ are   Suslin spaces.
\end{itemize}
\end{lemma}
\begin{proof}
In fact, $(\mathfrak{X}, d)$ is a metric space which is the product of two separable complete metric  spaces. Note that
$\mathscr{C}(I; B)$ is separable and complete with respect to the compact-open topology because $I$ is a hemicompact space (see {\it e.g.} \cite{MR0464128,MR0146625}). On the other hand, the identity map
$$
\mathrm{Id}: (\mathfrak X, d)  \to (\mathfrak X, d_*)\,,
$$
is continuous with $(\mathfrak X, d)$ a Polish space and hence its image $(\mathfrak X, d_*)$ is a Suslin space. The spaces  $(\mathscr{C}(I; B), d_0)$ and
$(\mathscr{C}(I; B), d_{0,*})$ are treated similarly.
\end{proof}

It is also  useful to stress the following result.
\begin{lemma}
\label{sec.glob.lem.2}
The $\sigma$-algebras of $(\mathfrak{X}, d)$ and  $(\mathfrak{X}, d_*)$ coincide.
\end{lemma}
\begin{proof}
See \cite[Lemma C.2]{alc2020}.
\end{proof}
Hence, as a consequence, the set of Borel probability measures  on $(\mathfrak{X}, d)$ and on $(\mathfrak{X}, d_*)$  coincide.

\subsection{Projective argument}\label{project}
We introduce the finite rank linear operators:
 \begin{equation}
 \begin{array}{rcl}
T_n: B  & \longrightarrow &  B \\
x & \longmapsto & T_n(x)=\displaystyle\sum_{k=1}^{n}  \langle x, e_k\rangle \, e^*_k.
\end{array}
\end{equation}

\begin{lemma}
\label{weaklyconv}
The operators $T_n: B\to B$ satisfy the following properties:
\begin{enumerate}[label=(\roman*)]
 \item $\Vert T_n(x)\Vert_{*} \leq  \Vert x \Vert_{*}, \quad \forall x \in B$.
\item  $\displaystyle\lim_{n \rightarrow +\infty} \Vert T_n(x)-x\Vert_{*}=0, \quad \forall x \in B.$
\end{enumerate}
\end{lemma}
\begin{proof}
Recall that the families $\{e_k\}_{k\in \NN}$ and $\{e^*_{k}\}_{k\in\NN}$ define  a fundamental and strongly total biorthogonal system satisfying  \ref{biorth.sys.1}-\ref{biorth.sys.3}  in Definition \ref{biorthogonal_system}.

\noindent
(i) For all $x\in B$,
\begin{eqnarray*}
\Vert T_n x\Vert_{*} &=&\sum_{k'\in \NN} \frac{1}{2^{k'}\Vert e_{k'}\Vert_{E}}  \, |\langle \sum_{k=1}^n \langle x,
e_k\rangle e^*_{k}, e_{k'}\rangle | \\
&=&\sum_{k=1}^n  \frac{1}{2^{k}\Vert e_{k}\Vert_{E}}  \, |\langle x,
e_k\rangle| \\
&\leq& \Vert x\Vert_{*}\,.
\end{eqnarray*}
(ii)   For all $x\in  {\rm { Span}}\{e^*_k,k\in\NN\}$, one checks
\begin{eqnarray*}
\lim_n \Vert T_n x-x\Vert_{*}=0.
\end{eqnarray*}
Then using  \textit{(i)} and the density of ${\rm { Span}}\{e^*_k,k\in\NN\}$ in $(B,\Vert\cdot\Vert)$,  we prove  \textit{(ii)} by an
approximation argument.
\end{proof}

\noindent  For  $n \in \NN$, let  $B_n=T_n(B)={\rm { Span}}(e_1^*,\dots,e_n^*)\subset B$.  Denote by
\[ \begin{array}{rcl}
P_n \, : B_n  & \longrightarrow &\RR^{n} \\
\displaystyle \sum_{k=1}^{n} a_k e^*_k & \longmapsto & \displaystyle (a_1, \cdots , a_{n}).
\end{array} \]
Define also the maps \[ \pi_n:B \longmapsto \RR^{n}, \quad  \pi_n=P_n\circ T_n,\]
 and
 \begin{equation}\label{tildepi}
 \begin{array}{rcl}
\tilde{\pi}_n \, : \RR^{n}  & \longrightarrow & B_n\subset B \\
(a_1,\cdots, a_{n}) & \longmapsto & \displaystyle \tilde{\pi}_n(a_1,\cdots, a_{n})=\sum_{k=1}^{n} a_k e_k^*.
\end{array}
\end{equation}
Remark that we have the following relations
\begin{equation}\label{relations} P_n \circ \tilde\pi_n={\rm id}_{\RR^{n}}, \qquad \tilde \pi_n \circ \pi_n=T_n.
\end{equation}

\medskip
Let  $(\mu_t)_{t\in I}$ be a family of Borel probability measures on $B$ satisfying the assumptions in Proposition \ref{propinfinite}.  Then consider the following image measures
\begin{equation}\label{imagemeasure}
\boxed{\displaystyle \mu_t^{n}:=(\pi_n)_{\sharp} \mu_t \in \mathscr{P}(\RR^{n})}\,, \quad \boxed{ \tilde\mu_t^{n}:=(\tilde \pi_n)_{\sharp} \mu_t^{n}=(T_n)_{\sharp} \mu_t \in \mathscr{P}(B_n)}\,.
\end{equation}
According to Definition \ref{chap3.def.wnar},  Lemma \ref{weaklyconv} implies
 \be \label{narrconv}
 \tilde{\mu}^n_t \underset{n\rightarrow\infty}{\rightharpoonup} \mu_t
 \ee
  for all $t \in I$.
In our case, it is useful to consider  $\RR^{n}$ with the norm
\begin{equation}\label{Rnorm}
\vert \vert \vert y \vert \vert \vert_{\RR^{n}}:= \| \tilde \pi_n y  \|_{*}.\,
\end{equation}
\begin{lemma}[Projection to finite dimensions]\label{chap.3.proj.arg}
Assume \eqref{chap3.s1.eq3}.
For each $n\in\NN $, the curve $(\mu^{n}_t)_{t\in I}$ given by \eqref{imagemeasure}  is narrowly continuous and satisfies  the statistical Liouville  equation,
\begin{equation}\label{weaksenseinR}
\frac{d}{dt} \int_{\RR^{n}}  \varphi (x)  \mu_t^{n}(dx) =\int_{\RR^{n}} \langle v^{n}(t,x), \nabla_{x} \varphi(x) \rangle_{\RR^{n}}
\,\mu_t^{n}(dx), \quad \forall \varphi \in \mathscr{C}_{c}^{\infty}(\RR^{n}),
\end{equation}
in the distribution sense over the interior of $I$ and for some   Borel vector field $v^{n}:I \times \RR^{n}\rightarrow  \RR^{n}$ satisfying
\begin{equation}\label{ass1}
\int_I \int_{\RR^{n}}  \vert \vert \vert v^{n}(t,x) \vert \vert \vert_{\RR^{n}} \,\mu_t^{n}(dx) \,\frac{dt}{\omega(|t|)}<+\infty.
\end{equation}

\end{lemma}
\begin{proof}
Notice first that the curve $(\mu^{n}_t)_{t\in I}$ is narrowly continuous. Indeed, let $\varphi \in \mathscr{C}_b(\RR^{n};\RR)$ implies $\varphi\circ \pi_n \in \mathscr C_b(B_w;\RR)$. Hence, the curve $t \longmapsto \int_{\RR^{n}} \varphi(x) \, d\mu^{n}_t(x)=\int_{B} \varphi\circ \pi_n(x) \, \mu_t(dx)$ is continuous since $(\mu_t)_{t\in I}$ is weakly narrowly continuous as it is (strongly) narrowly continuous (see Definitions
\ref{narrow_continuity} and \ref{chap3.def.wnar}).\\
We then consider the statistical Liouville equation (\ref{chap3.le}) and  select $F=\psi \circ \pi_n$ where $\psi \in \mathscr C_c^\infty (\RR^{n})$.
The left part of  (\ref{chap3.le}) is transformed to:
\be \label{eqn} \frac{d}{dt} \int_{B} F(x) \, \mu_t(dx)=\frac{d}{dt} \int_B \psi \circ \pi_n(x)\, \mu_t(dx)  =
\frac{d}{dt} \int_{\RR^{n}} \psi(y) \, d\mu_t^{n}(y).
\ee
Since  the differential ${\rm D}(\psi\circ \pi_n) (x)={\rm D}\psi (\pi_n(x))\circ   \pi_n \in \mathscr{L}(B,\RR)=E^{**}\supset E$,  identities with the element
$\displaystyle \sum_{j=1}^n \partial_j\psi(\pi_n(x)) \, e_j \in E$, the right part of  (\ref{chap3.le})  is transformed to:
\[\begin{aligned}
 \int_{B} \langle \nabla F (x), v(t,x) \rangle_{E,E^*} \, \mu_t(dx)&=  \int_{B} \langle \nabla\psi (\pi_n(x)),  \pi_n \circ  v(t,x) \rangle_{\RR^{n}} \, \mu_t(dx)\, .
 \end{aligned}
 \]
 Remark that  since  $(B,\|\cdot \|_{*})$ is a  separable Radon space, we can apply the disintegration Theorem \ref{disint} (see Appendices
 \ref{chap3.sec.appB} and \ref{chap3.sec.appD}). In particular,  there exists a $\mu^{n}_t$-a.e. determined family of measures $ \lbrace \mu^n_{t,y},\, y \in \RR^{n} \rbrace \subset \mathscr{P}(B)$ such that $\mu^n_{t,y}\big(B  \setminus(\pi_n)^{-1}(y) \big)=0$
 and  applying formula \eqref{f}  with $f(x)= \langle \nabla\psi (\pi_n(x)),  \pi_n \circ  v(t,x) \rangle_{\RR^{n}}$ yields
 \[\int_{B} \langle \nabla\psi (\pi_n(x)),  \pi_n \circ  v(t,x) \rangle_{\RR^{n}}\, \mu_t(dx)=\int_{\RR^{n}} \int_{(\pi_n)^{-1}(y)} \langle \nabla\psi (\pi_n(x)),  \pi_n \circ  v(t,x) \rangle_{\RR^{n}} \, \mu^n_{t,y}(dx) \, d\mu_t^{n}(y)\, .\]
 So, we get
 \[\begin{aligned}
 \int_{B} \langle \nabla F(x), v(t,x) \rangle_{E,E^*} \, \mu_t(dx) & =  \int_{\RR^{n}} \int_{(\pi_n)^{-1}(y)}\langle \nabla\psi (\pi_n(x)),  \pi_n \circ  v(t,x) \rangle_{\RR^{n}} \, \mu^n_{t,y}(dx) \, d\mu_t^{n}(y)
  \\& =  \int_{\RR^{n}} \langle \nabla\psi (y),  \int_{(\pi_n)^{-1}(y)}\pi_n \circ  v(t,x)\, \mu^n_{t,y}(dx) \rangle_{\RR^{n}}  \, d\mu_t^{n}(y)
  \\& =  \int_{\RR^{n}} \langle \nabla\psi (y),  v^{n}(t,y) \rangle_{\RR^{n}}  \, d\mu_t^{n}(y)\, .
\end{aligned}
\]
where we have introduced  the vector field $v^{n}$ as follows:
\begin{equation}\label{vd}
 v^{n}(t,y):=\int_{(\pi_n)^{-1}(y)}\pi_n \circ  v(t,x) \, \mu^n_{t,y}(dx), \text{ for $t\in I$ and  $ \mu^{n}_t$-a.e. $y\in \RR^{n}$.}
\end{equation}
Furthermore, by gathering the above  equations, we get
\[\frac{d}{dt} \int_{\RR^{n}} \psi(y) \, d\mu_t^{n}(y)= \int_{\RR^{n}} \langle \nabla\psi (y),  v^{n}(t,y) \rangle_{\RR^{n}}  \, d\mu_t^{n}(y)\,. \]
So, we obtain the  statistical  Liouville equation (\ref{weaksenseinR}). We also obtain that
\be \label{argument}
\begin{aligned}
&\int_{I}\int_{\RR^{n}} \vert \vert \vert v^{n}(t,y)\vert \vert \vert_{\RR^{n}} \, \mu^{n}_t(dy) \, \frac{dt}{\omega(|t|)} \\ & = \int_{I}\int_{\RR^{n}} \Big \vert \Big \vert \Big \vert\int_{(\pi_n)^{-1}(y)}\pi_n \circ  v(t,x) \, \mu^n_{t,y}(dx)\Big \vert \Big \vert \Big \vert_{\RR^{n}} \, \mu^{n}_t(dy) \, \frac{dt}{\omega(|t|)} \\ & \leq \int_{I}\int_{B} \|\tilde \pi_n\circ \pi_n \circ v(t,x) \|_{*}  \, \mu_t(dx) \,
\frac{dt}{\omega(|t|)}
\\ & \leq    \int_{I}\int_{B} \| v(t,x) \|  \, \mu_t(dx) \, \frac{dt}{\omega(|t|)}< +\infty,
\end{aligned}
\ee
where in the above lines, we have used \eqref{vd}, \eqref{Rnorm},  { the Disintegration Theorem \ref{disint}}, the second equality in (\ref{relations}),  Lemma \ref{weaklyconv}-(i) and  (\ref{chap3.s1.eq3}).
\end{proof}

\subsection{Analysis in finite dimensions}\label{finitedimension}
In this part, we restrict our selves to the case $B=\RR^{d}$.
So, we aim to prove the \emph{global superposition principle}  (Proposition \ref{propinfinite}) when $B=\RR^d$. In fact,
similar results are already known in  finite dimensions and proved  in the book of   Ambrosio et al. \cite{AmbrosioLuigi2005GFIM} and in the work of  Maniglia \cite{MR2335089}. The main difference here is  that we consider times in unbounded intervals like the half-line $\RR_+$ while in the latter references it is restricted to  $[0,T]$. Unfortunately, we  could not deduce directly Proposition \ref{propinfinite} in the case  $B=\RR^d$ from the result of Maniglia \cite{MR2335089} or even from \cite{AmbrosioLuigi2005GFIM}. In fact, one needs to go through the main ideas and adjust some topological arguments which are behind the compactness properties that lead to the construction of the probability measure $\eta$ on the  path space $\mathfrak X$.  We first discuss  in Subsection \ref{subsec2.1}  the case where the  vector field  is  locally Lipschitz in the second variable, then we consider the Borel case in Subsection \ref{subsec2.2}.

\subsubsection{The Lipschitz case}\label{subsec2.1}
Let $\|\cdot\|_{\RR^d}$ be any norm on $\RR^d$. We denote often $v_t=v(t,\cdot)$.
In this paragraph, we impose the following local  Lipschitz condition.
\begin{assumption}[Lipschitz condition]\label{lipass}
For  every compact set  $K \subset \RR^d$,
\begin{equation}\label{lipeqn}
 t\in I\mapsto\sup_{K}\Vert v_t\Vert_{\RR^d} +{\rm lip}(v_t,K)  \in L^1_{loc}(I,dt),
\end{equation}
where ${\rm lip}(\cdot,K)$ denotes the Lipschitz constant on $K$.
\end{assumption}

\begin{proposition}[The global superposition principle in the {\it Lipschitz case}]\label{prop1}
Consider $B=\RR^d$ and  $v$, $(\mu_t)_{t\in I}$ as in  Proposition  \ref{propinfinite} satisfying the same hypotheses. Additionally, assume that $v: I \times \RR^d \rightarrow \RR^d$ satisfies the local Lipschitz Assumption  \ref{lipass}. Then the conclusion of Proposition  \ref{propinfinite} holds true with the Borel probability measure $\eta$  defined as in \eqref{etad}.
\end{proposition}
\begin{proof}
  To prove the above result, we  recall in  Appendix \ref{chap3.sec.appA}  several auxiliary results from the literature.  In particular, the  proof  is based on Lemmas \ref{lemma2.1}--\ref{lemma2.2}.   The assumption  (\ref{locinteg}), which is equivalent to (\ref{chap3.s1.eq3}) by Remark \ref{chap.3.rem},  holds true. Thus, we can  apply Lemmas  \ref{lemma2.1} and \ref{lemma2.2} to show that there exists a  Borel set $\cG^{d,T}\subset \RR^d$ such that $\mu_{t_0}(\cG^{d,T})=1$ and for all $x\in \cG^{d,T}$, there is a unique solution $\gamma \in AC^1(I\cap[-T,T];\RR^d)$ to the initial value problem  (\ref{chap3.ivpinh}) on $I\cap [-T,T]$ while we take $T>\pm t_0$ so that $I\cap[-T,T]$ has non empty interior.

\medskip

Let $$\cG^d=\underset{T \in \NN, T>\pm t_0}{\bigcap} \cG^{d,T},$$ which is  a Borel subset of $\RR^d$. Remark that by construction $\lbrace \cG^{d,T}\rbrace_{T \in \NN}$  is a decreasing sequence of Borel sets. Then, by monotone convergence theorem, we have
\[\mu_{t_0}(\cG^d)=\mu_{t_0}\big(\underset{T \in \NN, T>\pm t_0}{\bigcap}\cG^{d,T}\big)=\lim_{T\rightarrow +\infty}\mu_{t_0}(\cG^{d,T})= 1.
 \]
Recall that in the Lipschitz case, we have uniqueness of solutions for the initial value problem \eqref{chap3.ivpinh}. So, we conclude that for each $x\in \cG^d$  there exists a global unique solution of  \eqref{chap3.ivpinh}. Thus,  we construct a well-defined
global flow
\[ \begin{array}{rcl}
\Phi \, :I \times \cG^d   & \longrightarrow & \RR^d \\
(t,x) & \longmapsto & \displaystyle \Phi_t(x)=\gamma_x(t),\,
\end{array} \]
where $\gamma_x$ is the global solution of  \eqref{chap3.ivpinh} with the initial condition $\gamma(t_0)=x$.
Note that $\Phi $ is the flow with prescribed initial conditions at time $t_0$.
Moreover, using the identity  \eqref{chap3.appA.eq.flow}, we have for all $ t\in I$,
\[ \mu_t=(\Phi_t)_{ \sharp}{\mu_{t_0}}\,.\]
Now, we construct the measure $\eta$ as
\begin{equation}\label{etad}
\eta=({\rm Id}\times \Phi_\cdot)_{\sharp}\mu_{t_0}\in \mathscr{P}(\mathfrak X)
\end{equation}
where ${\rm Id}\times \Phi_\cdot$ is the map
given by
\[ \begin{array}{rcl}
{\rm Id}\times\Phi_\cdot \, : \cG^d   & \longrightarrow &\mathfrak X=\RR^d  \times\mathscr{C}(I;\RR^d) \\
x & \longmapsto & \displaystyle (x,\Phi_\cdot(x))\equiv (x,\gamma_x(\cdot))\, .
\end{array} \]
Next, we want to prove that $\eta$ satisfies the conditions \ref{chap.3.(i)} and \ref{chap.3.(ii)} of Proposition \ref{propinfinite}.\\
{\bf For (i):} We have to prove that $\eta(\cF^d)=1$ where
\[\cF^d=\bigl\lbrace (x,\gamma) \in \mathfrak X;\quad  \gamma \in AC^1_{\it loc}(I;\RR^d), \ \gamma(t)=x +\int_{t_0}^{t}v(\tau,\gamma(\tau))\,d\tau, \ t \in I \bigr\rbrace\,.  \]
Indeed, we have for $t\in I$
\[\begin{aligned}
&\int_{\mathfrak X}  \left\lVert \gamma(t)-x-\int_{t_0}^{t} v(\tau,\gamma(\tau)) \, d \tau \right\rVert_{\RR^d} \,\eta(dx,d\gamma)\\ &=\int_{\mathfrak X} \left\lVert\gamma(t)-x-\int_{t_0}^{t} v(\tau,\gamma(\tau)) \,d \tau \right\rVert_{\RR^d}  \,({\rm Id}\times \Phi_\cdot)_{\sharp}\mu_{t_0}(dx,d\gamma)
 \\ &
=\int_{\RR^d} \left\lVert \Phi_t(x)-x-\int_{t_0}^{t} v(\tau,\Phi_\tau(x)) \,d \tau \right\rVert_{\RR^d}  \,\mu_{t_0}(dx)\\&=0\,,
 \end{aligned}
  \]
  which implies that we have
  \begin{equation}\label{duhamel}
   \gamma(t)=x+\int_{t_0}^{t} v(\tau,\gamma(\tau)) \,d \tau, \quad \eta-{\rm a.e.} \ (x,\gamma)\in\mathfrak X.
\end{equation}
{For each $t_j\in \mathbb{Q}\cap I$, a rational number, there exists a null set $\cN^{t_j}$ with $\eta(\cN^{t_j})=0$ and such that (\ref{duhamel}) holds true on  $\mathfrak X\setminus \cN^{t_j}$. Then taking
\begin{equation*} \cN= \underset{j \in \NN}{\bigcup} \cN^{t_j}
\end{equation*}
with $\eta(\cN)=0$ such that for all $(x,\gamma)  \notin \cN$ and for all rational numbers in $\{t_j\}_j=\QQ\cap I$, we have
\be  \label{duhameltj}\gamma(t_j)=x+\int_{t_0}^{t_j} v(\tau,\gamma(\tau)) \,d \tau\,.
\ee
Now, using the continuity of the curves $\gamma$ and since $v(\cdot,\gamma(\cdot))\in L^1_{loc}(I,dt)$ $\eta-$almost surely (see condition  \eqref{chap3.s1.eq3}), the identity \eqref{duhameltj} is well-defined and moreover we get (\ref{duhamel})  for all times $t \in I$ and for all $(x,\gamma) \notin \cN$. }\\
   {\bf For (ii):} Let $\varphi \in \mathscr{C}_b(\RR^d) $
 \[\begin{aligned}
\int_{\RR^d} \varphi(x)\;({(\Xi_t)}_{\sharp}\eta)(dx)&=\int_{\mathfrak X} \varphi(\Xi_t(x,\gamma)) \,\eta(dx,d\gamma)=
\int_{\mathfrak X} \varphi\circ \Xi_t (x,\gamma) \;({\rm Id}\times \Phi_\cdot)_{\sharp}\mu_{t_0}(dx,d\gamma)
 \\ &
=\int_{\cG^d} \varphi\circ \Xi_t \circ ({\rm Id}\times \Phi_\cdot)(x) \;\mu_{t_0}(dx)=\int_{\cG^d} \varphi  ( \Phi_t(x)) \,\mu_{t_0}(dx)
\\& =\int_{\RR^d} \varphi  ( x) \;(\Phi_t)_{\sharp}{\mu_{t_0}}(dx)=\int_{\RR^d} \varphi  ( x) \; \mu_t(dx).
 \end{aligned}
  \]
  And thus $(\Xi_t)_{\sharp}\eta=\mu_t$, for all $t \in I$.
\end{proof}

\subsubsection{The Borel case}\label{subsec2.2}
In this section, we prove  Proposition \ref{propinfinite} when $B=\RR^d$ and $v$ is a Borel vector field. Now since $v$ is no more assumed to be  Lipschitz in the  second variable, we have to take into account that the characteristics may not be unique. Indeed, the potential lack of uniqueness of solutions to (\ref{chap3.ivpinh}) on finite intervals  makes it impossible to follow the same strategy as before.

\begin{proposition}[The global superposition principle in the {\it Borel case}]\label{prop2}
Consider $B=\RR^d$ and  $v$, $(\mu_t)_{t\in I}$ as in  Proposition  \ref{propinfinite} satisfying the same hypotheses. Then the conclusion of Proposition  \ref{propinfinite} holds true.
\end{proposition}
\begin{proof}
The proof of Proposition \ref{prop2} is based on the  three lemmas stated below. The scheme goes as follows.
 We  apply first  the regularization Lemma \ref{approx}   to get an approximating family of probability measures $(\mu_t^\eps)_{t\in I}$ which  satisfies a statistical Liouville equation similar to \eqref{chap3.le} with a locally Lipschitz vector field $v^\eps$ satisfying Assumption \ref{lipass}. Then, we can apply  Proposition \ref{prop1} to the couple $(v^\eps,(\mu_t^\eps)_{t\in I})$ and  get a corresponding probability measure $\eta_{\eps} \in \mathscr{P}(\mathfrak X)$. Hence, we apply Lemma \ref{tight} to prove that the family $\lbrace \eta_{\eps} \rbrace_{\eps} $ is tight in $\mathscr{P}(\mathfrak{X})$. Therefore, there exists $\eta\in\mathscr P(\mathfrak X)$ such that $\eta_{\eps} \underset{\eps \rightarrow 0}{\rightharpoonup} \eta$ weakly narrowly (at least for a subsequence). Finally, by  Lemma \ref{limit}, we check that the constructed measure $\eta$ satisfies  \ref{chap.3.(i)} and \ref{chap.3.(ii)} in Proposition \ref{propinfinite}.
\end{proof}

We  provide here the above mentioned technical  Lemmas \ref{approx}, \ref{tight} and \ref{limit}.
\begin{lemma}[Regularization]\label{approx}
Consider $B=\RR^d$ and  $v$, $(\mu_t)_{t\in I}$ as in  Proposition  \ref{propinfinite} satisfying the same hypotheses. Then,  the regularized vector field  $v_t^{\eps}$  and the measures $\mu_t^{\eps}$ given in \eqref{approximation} satisfy a statistical Liouville equation as in \eqref{chap3.le}  over  the interval $I$. Moreover, define
\be \label{etafunction}\eta_{\eps}= ({\rm Id} \times \Phi_\cdot^\eps)_{\sharp}\mu_{t_0}^{\eps}\in \mathscr{P}(\mathfrak X)
\ee
 where  $\Phi_t^\eps(x) \equiv \gamma_x^\eps(t)$ is the unique global solution to the initial value problem
\begin{equation}\label{epseqn}
\dot{\gamma_x^{\eps}}(t)=v^{\eps}(t,\gamma_x^\eps(t)),\quad \gamma^\eps_x(t_0)=x.
\end{equation}
Then for all $t\in I$,
 \be
 \label{projeqmu}
 \mu^{\eps}_t=(\Xi_t)_{\sharp}\eta_{\eps}= (\Phi^\eps_t)_{\sharp}\mu_{t_0}^{\eps}.
 \ee
 \end{lemma}

\begin{proof}
We apply the regularization Lemma \ref{lemma2.6}, finding the approximation $\mu^{\eps}_t$  and $v^{\eps}_t$ in \eqref{approximation} satisfying the  Liouville equation \eqref{chap3.le}. In particular, the vector field $v^{\eps}_t$ is locally  Lipschitz as in Assumption \ref{lipass} and satisfies \eqref{chap3.s1.eq3} as a consequence of \eqref{cdt1}. Thus, we can apply Proposition \ref{prop1} to get  $\eta_{\eps}=({\rm Id} \times \Phi_\cdot^\eps)_{\sharp}\mu_{t_0}^{\eps} \in \mathscr{P}(\mathfrak X)$, with $\Phi_t^\eps(x) \equiv \gamma_x^\eps(t)$  the unique global solution to (\ref{epseqn}) for all $x\in \cG^{d,\eps}$,
where  we have denoted by $\cG^{d,\eps}$ the set of all initial data where (\ref{epseqn}) admits a unique global solution. More precisely, we define here
 \begin{equation}\label{Feps}
\cF^{d,\eps}:=\bigl\lbrace (x,\gamma) \in \mathfrak X; \gamma \in AC^1_{\it loc}(I;\RR^d), \ \gamma(t)=x +\int_{t_0}^{t}v^{\eps}(\tau,\gamma(\tau))d\tau, \   \forall t \in I \bigr\rbrace .
\end{equation}
We have then
\begin{equation}\label{Geps}
\cG^{d,\eps}=\bigl\lbrace x \in \RR^d ;\,  \exists   \gamma \in AC^1_{\it loc}(I;\RR^d) \, \text{ s.t. } \,  (x,\gamma) \in \cF^{d,\eps} \bigr\rbrace
\end{equation}
 with $\mu_{t_0}^{\eps}(\cG^{d,\eps})=1$ and $\eta_{\eps}(\cF^{d,\eps})=1$.
The identity \eqref{projeqmu} follows { from \eqref{etafunction}.}
\end{proof}

\begin{lemma}[Tightness]\label{tight}
The family $\lbrace \eta_{\eps} \rbrace_{\eps} $, defined in \eqref{etafunction}, is tight in $\mathscr{P}(\mathfrak X)$.
\end{lemma}
\begin{proof}
We use here Lemma \ref{lemma2.4} with $X \equiv \mathfrak X$, $X_1\equiv \RR^d$ and $X_2\equiv \mathscr{C}(I; \RR^d)$. The latter spaces are separable metric spaces. Recall that $\mathscr{C}(I; \RR^d)$ is endowed with the compact-open topology (see the metric $d$ in \eqref{chap.3.norm-m}-\eqref{chap.3.compact-open}).  Define the homeomorphism map
$r:=r^1\times r^2: X \rightarrow X $ by
\[\begin{array}{rcl}
{r^1} \, : X  & \longrightarrow & X_1 \\
(x,\gamma) & \longmapsto & \displaystyle r^1(x,\gamma)= x
\end{array} ,\qquad  \begin{array}{rcl}
{r^2} \, : X  & \longrightarrow & X_2 \\
(x,\gamma) & \longmapsto & \displaystyle r^2(x,\gamma)= \gamma-x\,.
\end{array} \]
It is obvious that $r$ is proper. To prove the tightness of $ \lbrace \eta_{\eps}\rbrace_\eps $, it suffices to prove:
\begin{enumerate}
\item  The family of measures $\lbrace (r^1)_{\sharp}\eta_{\eps} \rbrace_\eps$ is tight in $\mathscr{P}(\RR^d)$.
\vskip 1mm
\item   The family of measures $\lbrace (r^2)_{\sharp}\eta_{\eps} \rbrace_\eps$ is tight in $\mathscr{P}(\mathscr{C}(I; \RR^d))$.
\end{enumerate}
For (1), we have $  (r^1)_{\sharp}\eta_{\eps} =\mu_{t_0}^{\eps}$. Indeed, let $\varphi \in \mathscr C_b(\RR^d)$, we have
\[ \begin{aligned}
&\int_{\RR^d}\varphi(x)  \;(r^1)_{\sharp}\eta_{\eps}(dx)= \int_{\cF^{d,\eps}}\varphi(r^1(x,\gamma))\; \eta_{\eps}(dx,d\gamma)= \int_{\cF^{d,\eps}}\varphi(x) \, \eta_{\eps}(dx,d\gamma) \\ &=\int_{\cF^{d,\eps}}\varphi(\Xi_{t_0}(x,\gamma))\; \eta_{\eps}(dx,d\gamma)
=\int_{\cG^{d,\eps}}\varphi(x)\; (\Xi_{t_0})_{\sharp}\eta_{\eps}(dx)=\int_{\RR^d}\varphi(x)\;  \mu_{t_0}^{\eps}(dx),
\end{aligned}
\]
since $(\Xi_{t_0})_{\sharp}\eta_{\eps}=\mu_{t_0}^{\eps}$ and $\mu^{\eps}_{t_0}(\RR^d\setminus \cG^{d,\eps})=0$.
Remark also $\mu^{\eps}_{t_0} \underset{\eps \rightarrow 0}{\rightharpoonup} \mu_{t_0}$.
Since $\RR^d$ is a separable Radon space, we get by Lemma \ref{cam}, that the family $\lbrace (r^1)_{\sharp}\eta_{\eps} \rbrace_\eps$  is tight in $\mathscr{P}(\RR^d)$.

\medskip
The proof of (2) is more complicated to handle.  For that, we apply Lemma \ref{lemma2.5}.  In fact, by Lemma  \ref{remarkmeasure}, we get the existence of a non-decreasing  superlinear function $\theta:\RR_+\to[0,+\infty]$  satisfying the inequality \eqref{chap.3.est.theta}.
 Then we introduce $g: \mathscr C(I; \RR^d)\longrightarrow [0,+\infty]$
\begin{align*}
\notag
g(\gamma):= \begin{cases}
\displaystyle\int_{I} \theta(\Vert \dot{\gamma} \Vert_{\RR^d})\frac{dt}{\omega(|t|)}  \qquad &\text{if} \ \gamma(t_0)=0 \ \text{and}  \ \gamma \in AC^1_{\it loc}(I;\RR^d),\\
+\infty \qquad \qquad \qquad  &\text{if} \ \gamma(t_0)\neq 0 \ \text{or} \ \gamma \notin AC^1_{\it loc}(I;\RR^d).
\end{cases}
\end{align*}
In order to obtain the tightness of the family $\lbrace (r^2)_{\sharp}\eta_{\eps} \rbrace_\eps$,  it is enough  according to  Lemma \ref{lemma2.5} to prove the following points:
\begin{enumerate}[label=(\alph*)]
\item $\displaystyle \sup_{\eps>0} \int_{\mathscr C(I; \RR^d)} g(\gamma)\,\,(r^2)_{\sharp} \eta_{\eps}(d\gamma)<+\infty$.
\vskip 1mm
\item For all $c \geq 0$, the sublevel sets $\lbrace  \gamma \in\mathscr C(I; \RR^d); \  g(\gamma)\leq c \rbrace$  are {relatively} compact in the space  $\mathscr C(I; \RR^d)$ endowed with the compact-open topology.
\end{enumerate}
For (a), let $\eps>0$. We have for  $\cF^{d,\eps}$ and $\cG^{d,\eps}$ as defined in (\ref{Feps}) and (\ref{Geps})
\[ \begin{aligned}
\int_{\mathscr C(I; \RR^d)}g(\gamma) \,\,  (r^2)_{\sharp} \eta_{\eps}(d\gamma) &= \int_{\mathscr C(I; \RR^d)}
 \int_{I} \theta(\Vert \dot{\gamma}(t) \Vert_{\RR^d})\frac{dt}{\omega(|t|)} \,   \, (r^2)_{\sharp} \eta_{\eps}(d\gamma)
\\ &=\int_{I} \int_{\cF^{d,\eps}}   \theta(\Vert v^{\eps}(t,{\gamma}(t)) \Vert_{\RR^d}) \,  \,   \eta_{\eps}(dx,d\gamma)
 \, \frac{dt}{\omega(|t|)}
 \\ & =  \int_{I} \int_{\cG^{d,\eps}}  \theta\Big(\Vert v^{\eps}(t,{\Phi}_t^\eps(x)) \Vert_{\RR^d}\Big)  \,    \mu_{t_0}^{\eps}(dx) \, \frac{dt}{\omega(|t|)}
\\ & \leq \int_{I} \int_{\RR^d}  \theta(\Vert v^{\eps}(t,x) \Vert_{\RR^d})    \,    \mu_t^{\eps}(dx)\, \frac{dt}{\omega(|t|)}\\ &  \leq \int_{I} \int_{\RR^d}  \theta( \Vert v(t,x) \Vert_{\RR^d})     \,  \mu_t(dx)  \, \frac{dt}{\omega(|t|)}<+\infty.
\end{aligned}
\]
 Note that for the above inequality, we used Fubini's theorem and in the last step, we used  \cite[Lemma 3.10]{{MR2335089}} which generalizes \eqref{cdt1}.  The above inequality holds uniformly in $\eps>0$.

\noindent
For (b),  thanks to Lemma \ref{localtight} and Remark \ref{chap.3.rem.fincomp}, the sublevels
$$
\cA_c=\lbrace  \gamma \in\mathscr C(I; \RR^d); \  g(\gamma)\leq c \rbrace
$$
are  {relatively} compact in the separable metric space $(\mathscr C(I; \RR^d),d_0)$ with the distance $d_0$ inducing the compact-open topology   given in \eqref{chap.3.compact-open}. Hence, thanks to Lemma \ref{lemma2.4}, we conclude that $\lbrace  \eta_{\eps}\rbrace_\eps $ is tight in  $\mathscr{P}(\mathfrak X)$.

\end{proof}

\begin{lemma}[Concentration and lifting properties]\label{limit}
The subsequential limit $\eta$  (in the sense of narrow convergence) of the family $\lbrace  \eta_{\eps}\rbrace_\eps $  satisfies  \ref{chap.3.(i)} and \ref{chap.3.(ii)} in Proposition \ref{propinfinite}.
\end{lemma}
\begin{proof}
The existence of $\eta$ is guaranteed by Lemma \ref{tight}.
We show that $\eta$ satisfies \ref{chap.3.(i)} and \ref{chap.3.(ii)} in Proposition \ref{propinfinite}.
 Remark that we do not have the explicit expression for  $\eta$ in this case. And thus, we cannot proceed as before to prove first \ref{chap.3.(i)}. Then, we start with proving \ref{chap.3.(ii)}. For $\varphi \in \mathscr C_b(\RR^d;\RR)$, we have
\[\begin{aligned}
\int_{\RR^d} \varphi(x)  \,\mu^{\eps}_t(dx)=\int_{\mathfrak X} \varphi(\gamma(t)) \,\, \eta_{\eps}(dx,d\gamma)\,,
\end{aligned}
\]
where  $\eta_{\eps}$ is as in \eqref{etafunction}. Hence, we can let $\eps \rightarrow 0$ in the above equation and deduce
\[\begin{aligned}
\int_{\RR^d} \varphi(x)  \,\mu_t(dx)=\int_{\mathfrak X} \varphi(\gamma(t)) \,\, \eta (dx,d\gamma).
\end{aligned}
\]
The above equality is true for all $\varphi \in \mathscr C_b(\RR^d;\RR)$. This gives $(\Xi_t)_{\sharp} \eta= \mu_t$, for all $t \in I.$ And thus, condition \ref{chap.3.(ii)} is satisfied.
Finally, we check condition \ref{chap.3.(i)}. Let $ w:I \times \RR^d \rightarrow \RR^d$ be a bounded continuous vector field. We write  $w(t,x)\equiv w_t(x)$
and introduce  the regularized  vector field  $\displaystyle w^{\eps}_\tau:=\frac{(w_\tau \mu_\tau)\ast \rho_\eps}{\mu_\tau\ast \rho_\eps}$ (for $\rho_\eps$ as in Lemma \ref{lemma2.6}).
For $T>0$, we have for all $t\in J=I\cap[-T,T]$,
\[ \begin{aligned}
&\int_{\mathfrak X} \left\lVert \gamma(t)-x-\int_{t_0}^{t} w(\tau,\gamma(\tau)) \,d \tau \right\rVert_{\RR^d} \,\,\eta_{\eps}(dx,d\gamma)
 \\ &
=\int_{\cF^{d,\eps}}  \left\lVert \int_{t_0}^{t}v^{\eps}(\tau, \gamma(\tau)) \ d\tau -\int_{t_0}^{t} w(\tau,\gamma(\tau)) \,d \tau  \right\rVert_{\RR^d} \,\,\eta_{\eps}(dx,d\gamma)\\&
=\int_{\cG^{d,\eps}} \left\lVert \int_{t_0}^{t} v^{\eps}(\tau,\Phi^\eps_\tau(x))\, d\tau-\int_{t_0}^{t} w(\tau,\Phi^\eps_\tau(x)) \,d \tau \right\rVert_{\RR^d} \,\mu_{t_0}^{\eps}(dx)\\&
\leq
 \int_{J}\int_{\RR^d} \Vert v^{\eps}(\tau,x)- w(\tau,x)\Vert_{\RR^d}   \,\mu_\tau^{\eps}(dx)\,d \tau\\&
\leq
 \int_{J}\int_{\RR^d} \Vert v^{\eps}(\tau,x)- w^{\eps}(\tau,x)\Vert_{\RR^d}  \,\mu_\tau^{\eps}(dx)\,d \tau
 \\ & \quad +\int_{J}\int_{\RR^d}  \Vert w^{\eps}(\tau,x)- w(\tau,x)\Vert_{\RR^d}  \,\mu_\tau^{\eps}(dx)\,d \tau
 \\&
\leq \int_{J} \int_{\RR^d} \Vert v(\tau,x)- w(\tau,x)\Vert_{\RR^d}   \,d\mu_\tau(x)\,d \tau
\\ & \quad +  \int_{J}\int_{\RR^d} \left\lVert (w_\tau  \mu_\tau) \ast \rho_{\eps}(x)- w(\tau,x) \ \mu_\tau \ast \rho_{\eps}(x)\right\rVert_{\RR^d} \mu_\tau(dx) \, d\tau
\\&
\leq \int_{J} \int_{\RR^d} \left\lVert v(\tau,x)- w(\tau,x)\right\rVert_{\RR^d}   \,\mu_\tau(dx)\,d \tau
\\ & \quad +  \int_{J}\int_{\RR^d} \Big[ \int_{\RR^d} \left\lVert w(\tau,x)- w(\tau,y)  \right\rVert_{\RR^d} \rho_{\eps}(x-y) dx \Big] \, \mu_\tau(dy) \,d \tau
 .
\end{aligned}
\]
For the above inequality, we used Fubini's theorem as well as the inequality
(\ref{cdt1}).
On the other hand,
\[ \begin{aligned}
&\int_{J}\int_{\RR^d} \Big[ \int_{\RR^d} \left\lVert w(\tau,x)- w(\tau,y)  \right\rVert_{\RR^d} \rho_{\eps}(x-y)\,  dx \Big] \, \mu_\tau(dy) \,d \tau \\ & \leq
 \int_{J}\int_{\RR^d} \Big[ \int_{\RR^d} \Vert w(\tau,\eps \tilde x+y)- w(\tau,y)  \Vert_{\RR^d} \, \rho(\tilde x)\,  d\tilde x \Big] \, \mu_\tau(dy) \,d \tau
 \end{aligned}\]
where $\rho_\eps(x)=1/\eps^d \, \rho(x/\eps)$. Moreover, by the  Lebesgue dominated convergence theorem, the above expression tends to zero as $\eps \rightarrow 0$.   We then deduce that
\begin{equation}
\label{w_bound}
\int_{\mathfrak X} \left\lVert\gamma(t)-x-\int_{t_0}^{t} w(\tau,\gamma(\tau)) \,d \tau \right\rVert_{\RR^d} \,\,\eta(dx,d\gamma)
\leq  \int_{J}\int_{\RR^d} \Vert v(\tau,x)- w(\tau,x)\Vert_{\RR^d}   \,\mu_\tau(dx)\,\,d \tau.
\end{equation}
{Follow the same argument as  in the proof of  Lemma  \ref{remarkmeasure} to define a measure $\nu$ as in \eqref{chap.3.def.nu} on the product space $J \times \RR^d$. Then using Lemma \ref{density},  the space $\mathscr C_b(J \times \RR^d ;\RR^d)$ is dense in $L^1(J\times\RR^d,\nu;\RR^d)$. Let $(w_m)_m\subset \mathscr C_b(J \times \RR^d ;\RR^d)$ be a sequence of continuous bounded functions  converging to $v \in L^1(J\times\RR^d,\nu;\RR^d)$.}
 Using the fact that $\mu_t=(\Xi_t)_{\sharp} \eta$ as well as $\lim_m {w_m} = v$ in $L^1(J\times\RR^d,\nu;\RR^d)$, we have  for all $T>0$

\begin{equation}\label{con}
\begin{aligned} &\int_{\mathfrak X} \int_{J}\Vert w_m(\tau,\gamma(t))- v(\tau,\gamma(t))\Vert_{\RR^d} \,d \tau  \,\,\eta(dx,d\gamma)\\& =
\int_{J}  \int_{\RR^d} \Vert w_m(\tau,x)- v(\tau,x)\Vert_{\RR^d}   \,\mu_\tau(dx)\,d \tau
 \underset{m\rightarrow +\infty}{\longrightarrow} 0 \, .
\end{aligned}
\end{equation}
 At the end, we  use the triangle inequality, and apply  \eqref{w_bound} with  $w\equiv w_m$ together with (\ref{con}) to obtain
\[
\begin{aligned}
0&\leq \int_{\mathfrak X} \left\lVert\gamma(t)-x-\int_{t_0}^{t} v(\tau,\gamma(\tau)) \,d \tau \right\rVert_{\RR^d} \,\,\eta(dx,d\gamma)
 \\ &
\leq \int_{\mathfrak X} \left\lVert\gamma(t)-x-\int_{t_0}^{t} w_m(\tau,\gamma(\tau)) \,d \tau \right\rVert_{\RR^d} \,\,\eta(dx,d\gamma)+\int_{\mathfrak X}\int_{J} \Vert (w_m- v)(\tau,\gamma(\tau))\Vert_{\RR^d} \,d \tau  \,\,\eta(dx,d\gamma)
\\ &
\leq   2\int_{J}\int_{\RR^d} \Vert v(\tau,x)- w_m(\tau,x)\Vert_{\RR^d}   \,\mu_\tau(dx)\,d \tau.
\end{aligned}
 \]
 We take $m\rightarrow +\infty$  to deduce that:
 $\forall T\in \RR_+^*, \ \forall t \in J= I\cap  [-T,T]$,
  \begin{equation}\label{concen}
  \int_{\mathfrak X} \left\lVert \gamma(t)-x-\int_{t_0}^{t} v(\tau,\gamma(\tau)) \,d \tau \right\rVert_{\RR^d} \,\,\eta(dx,d\gamma)=0.
  \end{equation}
Hence,  for all $t \in I$, we get \eqref{concen}.
This implies that for each $t \in I$, we have
\begin{equation}\label{appro}\gamma(t)=x+\int_{t_0}^{t} v(\tau,\gamma(\tau)) \,d \tau, \ {\rm for}  \ \eta-a.e . \ (x,\gamma)\in \RR^d\times \mathscr{C}(I; \RR^d).
\end{equation}
{ Now, due to the continuity of the curves $\gamma$ as well as  $v(\cdot,\gamma(\cdot))\in L^1_{\it loc}(I;dt)$ $\eta-$almost surely, using the  same arguments as in the proof of Proposition \ref{prop1}, we can find by density arguments  an $\eta-$null set $\cN$ such that the Duhamel formula \eqref{appro} holds true for all times $t \in I$ and for all $(x,\gamma) \notin \cN$.}
\end{proof}

\subsection{Analysis on Banach spaces}\label{proofofprop2}
We want to complete the  proof of the \emph{global superposition principle} of  Proposition \ref{propinfinite} by applying the results in the previous Section \ref{finitedimension}.

\medskip
\noindent
\emph{Proof of Proposition \ref{propinfinite}:}
The strategy of the proof is similar to the finite-dimensional case in Proposition \ref{prop2}.  Consider $(B,\Vert\cdot\Vert)$ to be an infinite-dimensional
separable dual Banach space. Let $v$ and  $(\mu_t)_{t\in I}$ as in  Proposition  \ref{propinfinite} satisfying the same hypotheses.
Recall the image measures  $\mu^n_t\in\mathscr{P}(\RR^n)$ and  $\tilde\mu_t^n\in \mathscr{P}(B_n)$ given in  \eqref{imagemeasure} and the subspace $B_n={\rm Span}(e^*_1,\dots,e^*_n)\subset B$ as well as  $\vert \vert \vert \cdot \vert \vert \vert_{\RR^d}$ in \eqref{Rnorm}.
We  apply then  the projection argument in Lemma  \ref{chap.3.proj.arg}. Hence, we conclude that there exists a Borel vector field
$v^n:I\times \RR^n\to\RR^n$ given in \eqref{vd} such that the probability measures $(\mu^n_t)_{t\in I}$ satisfy the estimate \eqref{ass1} and the statistical Liouville equation \eqref{weaksenseinR}. Therefore, we have all the ingredients to apply Proposition \ref{prop2} for the couple
$(v^n_t, (\mu^n_t)_{t\in I})$ and get the existence of the path measure $\eta^n\in\mathscr{P}(\RR^n\times \mathscr{C}(I;\RR^n))$ so that  $\eta^n$ satisfies the concentration and lifting properties in Proposition  \ref{prop2}  for each $n\in\NN$. \\
We then define
\begin{equation}\label{tildeeta}
 \boxed{\tilde{\eta}^{n} := (\tilde\pi_n \times \tilde \pi_n)_{\sharp} \eta^{n} \in \mathscr{P}(\mathfrak{X}_n)},
\end{equation}
where $\tilde{\pi}_n$ is introduced in \eqref{tildepi} and  $\mathfrak{X}_n:= B_n \times \mathscr{C}(I;B_n)\subset \mathfrak{X}=B \times \mathscr{C}(I;B)$.
Thanks to Lemma \ref{chap.3.tight.B} given below, we  obtain that the sequence $\{\tilde\eta^n\}_{n}$ is tight in
$ \mathscr{P}(\mathfrak{X})\supset \mathscr{P}(\mathfrak{X}_n)$ (since $\mathfrak{X}_n$ is a Borel subset of $\mathfrak{X}$).
So,  there exist an $\eta\in\mathscr P(\mathfrak X)$  and a  subsequence that we still denote by $(\tilde\eta^{n} )_n$ such that $\tilde\eta^{n} \underset{ n\rightarrow \infty}{\rightharpoonup} \eta$ weakly narrowly. Finally, by  Lemma \ref{chap.3.i+ii}, we conclude that the constructed path measure $\eta$ satisfies  \ref{chap.3.(i)} and \ref{chap.3.(ii)} in Proposition \ref{propinfinite}.

\hfill$\square$

 We are now going to state and prove the aforementioned  technical Lemmas \ref{chap.3.tight.B} and \ref{chap.3.i+ii} used in the proof of Proposition \ref{propinfinite}.
\begin{lemma}[Tightness in Banach spaces]
\label{chap.3.tight.B}
The family of path measures $\lbrace \tilde{\eta}^{n} \rbrace_{n} $ given in \eqref{tildeeta},  is tight  for the weak narrow topology of $\mathscr{P}(\mathfrak X)$.
\end{lemma}
\begin{proof}
We use here Lemma \ref{lemma2.4} with $X\equiv (\mathfrak X,d_*)$ defined in \eqref{mathcalH}, $X_1\equiv B_w=(B,\Vert\cdot\Vert_*)$ and $X_2\equiv (\mathscr{C}(I;B), d_{0,*})$. The latter spaces are separable metric spaces. Define the homeomorphism map
$r:=r^1\times r^2: X\rightarrow X $ by
\[\begin{array}{rcl}
{r^1} \, : X  & \longrightarrow & X_1 \\
(x,\gamma) & \longmapsto & \displaystyle r^1(x,\gamma)= x
\end{array} ,\qquad  \begin{array}{rcl}
{r^2} \, : X  & \longrightarrow & X_2 \\
(x,\gamma) & \longmapsto & \displaystyle r^2(x,\gamma)= \gamma-x\,.
\end{array} \]
According to Lemma \ref{lemma2.4}, to prove the tightness of $ \lbrace \tilde{\eta}^{n}\rbrace_n $, it suffices to show:
\begin{enumerate}
\item  The family of measures $\lbrace (r^1)_{\sharp}\tilde{\eta}^{n} \rbrace_n$ is tight in $\mathscr{P}(X_1)$.
\vskip 1mm
\item   The family of measures $\lbrace (r^2)_{\sharp}\tilde{\eta}^{n} \rbrace_n$ is tight in $\mathscr{P}(X_2)$.
\end{enumerate}
For (1), we have $  (r^1)_{\sharp}\tilde\eta^{n} =\tilde\mu_{t_0}^n$ for all $n \in \NN$. Indeed, let $\varphi \in \mathscr C_b(B_w;\RR)$, we have
\[ \begin{aligned}
&\int_{B}\varphi(x)  \,\,(r^1)_{\sharp}\tilde\eta^{n}(dx)= \int_{\mathfrak X}\varphi(r^1(x,\gamma))\,\, \tilde\eta^{n}(dx,d\gamma)= \int_{\mathfrak X}\varphi(x) \, \,\tilde\eta^{n}(dx,d\gamma)  =\int_{\mathfrak X_n}\varphi(\tilde \pi_n(x)) \,\, \eta^{n}(dx,d\gamma) \\ & =\int_{\mathfrak X_n}\varphi\circ \tilde \pi_n \circ \Xi_{t_0}(x,\gamma) \,\, \eta^{n}(dx,d\gamma)=\int_{\RR^{n}}\varphi\circ \tilde \pi_n (x) \,  \mu^{n}_{t_0}(dx)
=\int_{B} \varphi(x) \,  \tilde\mu_{t_0}^n(dx).
\end{aligned}
\]
where $(\tilde\pi_{n})_{ \sharp}\mu^{n}_{t_0}=\tilde\mu_{t_0}^n$ is given in  \eqref{imagemeasure}.
Remark also $\tilde{\mu}^{n}_{t_0}\rightharpoonup \mu_{t_0}$  as $n \rightarrow +\infty$.
And since $B_w$ is a separable Radon space, we get by Lemma \ref{cam}, that the family $\lbrace (r^1)_{\sharp}\tilde\eta^{n} \rbrace_n$  is tight in $\mathscr{P}(B)$.

\smallskip\noindent
The proof of (2)  follows the same strategy  as in the finite-dimensional case.
Using Lemma \ref{remarkmeasure}, there exists a non-decreasing super-linear continuous convex function $\theta: \RR_+\rightarrow [0,+\infty]$  such that
\be \label{estimating}
\int_{I} \int_{B}  \theta(\Vert v(t,x) \Vert_{B}) \, \mu_t(dx)\, \frac{dt}{\omega(|t|)} \leq 1.
\ee
 We want to apply Lemma \ref{lemma2.5}. To this end, we introduce $g: \mathscr{C}(I;B)\longrightarrow [0,+\infty]$
\begin{align*}
\notag
g(\gamma):= \begin{cases}
\displaystyle\int_{I} \theta(\Vert \dot{\gamma} \Vert_{*}) \, \frac{dt}{\omega(|t|)}  \qquad &\text{if} \ \gamma(t_0)=0 \ \text{and}  \ \gamma \in AC^1_{{\it loc}}(I;B),\\
+\infty \qquad \qquad \qquad  &\text{if} \ \gamma(t_0)\neq 0 \ \text{or} \ \gamma \notin AC^1_{\it loc}(I;B).
\end{cases}
\end{align*}
According to  Lemma \ref{lemma2.5}, we have to prove the following points:
\begin{enumerate}[label=(\alph*)]
\item $\displaystyle \sup_{n\in \NN} \int_{ \mathscr{C}(I;B)} g(\gamma)\, \,(r^2)_{\sharp} \tilde\eta^{n}(d\gamma)<+\infty$.
\vskip 1mm
\item The sublevel sets $\cA_c:=\lbrace  \gamma \in  \mathscr{C}(I;B); \  g(\gamma)\leq c \rbrace$  are {relatively} compact in $ \big( \mathscr{C}(I;B),d_{0,*}\big)$ for all $c \geq 0$.
\end{enumerate}
For (a), let $n \in \NN$. We have
\[ \begin{aligned}
\int_{\mathscr{C}(I;B)}g(\gamma) \,  \,(r^2)_{\sharp} \tilde\eta^{n}(d\gamma) &= \int_{\mathscr{C}(I;B)}  \int_{I} \theta(\| \dot{\gamma}(t) \|_{*}) \, \frac{dt}{\omega(|t|)} \, \,   (r^2)_{\sharp} \tilde\eta^{n}(d\gamma)\\ & = \int_{\mathfrak X}  \int_{I} \theta(\| \dot{\gamma}(t) \|_{*}) \, \frac{dt}{\omega(|t|)} \,    \, \tilde\eta^{n}(dx,d\gamma) \\ & =\int_{\RR^n \times \mathscr{C}(I ; \RR^n)}  \int_{I} \theta(\| \tilde{\pi}_n(\dot{\gamma}(t) )\|_{*}) \, \frac{dt}{\omega(|t|)} \, \,    \eta^{n}(dx,d\gamma) \\& =\int_{\RR^n \times \mathscr{C}(I ; \RR^n)}  \int_{I} \theta(\| \tilde{\pi}_n\circ v^{n}(t,\gamma(t)) \|_{*}) \, \frac{dt}{\omega(|t|)} \,   \,  \eta^{n}(dx,d\gamma) \\&
= \int_{I} \int_{\RR^n \times \mathscr{C}(I ; \RR^n)}   \theta(\vert \vert \vert  v^{n}(t,\gamma(t) )\vert \vert \vert_{\RR^{n}})  \, \,   \eta^{n}(dx,d\gamma) \, \frac{dt}{\omega(|t|)} \\&
= \int_{I}\int_{\RR^{n}}  \theta\Bigg(\left \vert \left\vert \left\vert \int_{( \pi_n)^{-1}(y)} \pi_n \circ v(t,x ) \, \mu^n_{t,y}(dx)\right \vert \right \vert \right \vert_{\RR^{n}}\Bigg)  \,     \mu^{n}_t(dy)\, \frac{dt}{\omega(|t|)}
 \\ & \leq
 \int_{I} \int_{B}  \theta( \| v(t,x)\|_{*}) \,     \mu_t(dx) \,\frac{dt}{\omega(|t|)},
\end{aligned}
\]
where we used the definition of $v^{n}$ as in \eqref{vd}, Jensen's inequality and the arguments as for \eqref{argument}.
 For the last line, we used
\begin{equation*}
\left \vert \left \vert \left \vert \pi_n \circ v(t,x) \right \vert \right \vert \right \vert_{\RR^n}= \|\tilde \pi_n \circ \pi_n \circ v(t,x)\|=\|T_n v(t,x)\|_{*} \leq \|v(t,x)\|_{*}\,,
\end{equation*}
 which follows from \eqref{Rnorm}, the second equality in \eqref{relations}, and Lemma \ref{weaklyconv} (i).
Then by using the estimate \eqref{estimating}, the monotonicity of $\theta$, and $\|v(t,x)\|_{*} \leq \|v(t,x)\|$ (which follows from \eqref{chap3.norm.def}), we conclude from the above calculation that
\[ \displaystyle \sup_{n\in \NN} \int_{ \mathscr{C}(I;B)} g(\gamma)\,\, (r^2)_{\sharp} \tilde\eta^{n}(d\gamma)<+\infty.  \]
For (b), we apply Lemma \ref{localtight} and conclude that the sublevels $\cA_c:=\lbrace  \gamma \in  \mathscr{C}(I;B); \  g(\gamma)\leq c \rbrace$  are relatively compact in $ \big( \mathscr{C}(I;B_w),d_{0,*}\big)$ for all $c \geq 0$.  However, we still need to check that  $\cA_c$ is relatively compact in $\big(\mathscr{C}(I;B), d_{0,*}\big)$. Let $(\gamma_n)_n $ be a sequence in $ \cA_c$. Then there is a subsequence $(\gamma_{n_k})_{k}$ and $\gamma \in\mathscr{C}(I;B_w)$ such that $d_{0,*}(\gamma_{n_k};\gamma) \underset{k\rightarrow \infty}{\longrightarrow 0}$. Hence, we just need  to prove that  $\gamma \in \mathscr{C}(I;B)$. First remark that
\[\|\gamma_n(t)-\gamma_n(t_0)\|_B \leq \int_{[t_0,t]} \| \dot{\gamma}_n(s)\|_B \, ds . \]
Assume $|t| \leq T$ and consider the set
 \[\cF=\lbrace  \| \dot{\gamma}_n(\cdot)\|_B ; \, n\in \NN \rbrace . \]
Remark that
\be
\displaystyle \cF \subset \Big \lbrace  f \in L^1 ( I;\frac{dt}{\omega(|t|)}): \, \int_{I} \frac{\theta(|f|)}{c} \, \frac{ds}{\omega(|s|)} \leq 1 \Big\rbrace .
\ee
Then by Lemma \ref{equivalence}, $\cF$ is equi-integrable. And thus by the Dunford-Pettis theorem \ref{Dunford-Pettis}, $\cF$ is relatively sequentially compact in the topology $\sigma(L^1, L^\infty)$. More precisely, this means there exists $  m(\cdot) \in L^1(I, \frac{dt}{\omega(|t|)})$ such that
\[\| \dot{\gamma}_n(\cdot)\|_B \rightharpoonup m(\cdot)\, \, \text{in} \, \,  \sigma(L^1, L^\infty) .\]
And thus, for all $f \in L^\infty$
 \[ \int_{I}\| \dot{\gamma}_n(s)\|_B f(s) \, \frac{ds}{\omega(|s|)} \longrightarrow \int_{I} m(s) f(s)\, \frac{ds}{\omega(|s|)} .\]
Take $f=1$, we get
\[\int_{I} \| \dot{\gamma}_n(s)\|_B \, \frac{ds}{\omega(|s|)} \longrightarrow \int_{I} m(s) \, \frac{ds}{\omega(|s|)}.
\]
Since $\Vert\cdot\Vert_*$ induces  the weak-* topology on bounded sets, for $|t|,|t_0|\leq T$, we get
\[ \begin{aligned}
\|\gamma(t)-\gamma(t_0)\|_B & \leq  \liminf_{n\rightarrow +\infty} \| \gamma_n(t)-\gamma_n(t_0)\|_B \\& \leq \liminf_{n\rightarrow +\infty} \int_{[t_0,t]} \| \dot{\gamma}_n(s)\|_B \, ds \\& \lesssim \, \liminf_{n\rightarrow +\infty} \int_{[t_0,t]} \| \dot{\gamma}_n(s)\|_B \frac{ds}{\omega(|s|)}
\\& \lesssim \,\int_{[t_0,t]} \frac{ m(s)}{\omega(|s|)} \, ds.
\end{aligned} \]
Thus, we conclude that $\gamma \in AC^1_{\it loc}(I;B)\subset \mathscr{C}(I;B)$.
\end{proof}

We give now the proof of the concentration and lifting properties \ref{chap.3.(i)} and \ref{chap.3.(ii)}  in Proposition \ref{propinfinite}.
\begin{lemma}\label{chap.3.i+ii}
Let $\eta\in\mathscr{P}(\mathfrak X)$ be any cluster point of the tight sequence $\{\tilde\eta^n\}_{n}$ defined in  \eqref{tildeeta}. Then $\eta$ satisfies the properties \ref{chap.3.(i)} and \ref{chap.3.(ii)} of Proposition  \ref{propinfinite}.
\end{lemma}
\begin{proof}
We start to give the proof of \ref{chap.3.(ii)}. Then, we address to the proof of (i) which can be achieved using \ref{chap.3.(ii)}.\\
{\bf For (ii):}\\
We have, for $n \in \NN$, for $\varphi \in \mathscr{C}_b(B_w)$, by \eqref{tildeeta}, Proposition \ref{prop2} and \eqref{imagemeasure}
\[\begin{aligned} \int_{\mathfrak X} \varphi(\gamma(t))\, \,\tilde{\eta}^n(dx,d\gamma)& = \int_{\RR^n \times \mathscr{C}(I ; \RR^n)} \varphi \circ \tilde{\pi}_n(\gamma(t)) \, \,\eta^{n} (dx,d\gamma)=\int_{\RR^n} \varphi(\tilde{\pi}_n(x)) \, \mu^{n}_t(dx)\\ &=\int_{B} \varphi(x) \, (\tilde{\pi}_{n})_{\sharp}\mu^{n}_t(dx)=\int_{B} \varphi(x) \, \tilde{\mu}_t^n(dx).
\end{aligned}\]
By   \eqref{imagemeasure}, we have $\tilde{\mu}^n_t \rightharpoonup \mu_t$  and $ \tilde{\eta}^n \rightharpoonup \eta$  as $n \rightarrow +\infty$. We take limits  in the above formula to get
\[\int_{\mathfrak X} \varphi(\gamma(t))\,\, {\eta}(dx,d\gamma) =\int_{B} \varphi(x) \, {\mu}_t(dx). \]
The above equality implies that ${\Xi_t}_{\sharp}\eta =\mu_t,$ for all $t \in I$. \\
{\bf For (i):}
\\
Let $h:I \times B\rightarrow  B$ be any bounded continuous function. Then, we claim that  for all $t\in I$
\begin{equation}\label{bdd}
\int_{\mathfrak X} \left\lVert \gamma(t)-x-\int_{t_0}^{t}h(\tau,\gamma(\tau)) \, d\tau \right\rVert_{*} \, \eta(dx,d\gamma)\leq \int_{[t_0,t]} \int_{B} \Vert v(\tau,x) -h(\tau,x)\Vert_* \, \mu_\tau(dx) \, d\tau.\,
\end{equation}
 To prove the above inequality, we consider first the following points:
\begin{enumerate}
\item Using the disintegration Theorem  \ref{disint} with the projection $T_n: B \rightarrow B_n$, and since $\tilde{\mu}_t^n=({T_n})_{\sharp}\mu_t \in \mathscr{P}(B_n)$, there exists a $\tilde{\mu}_t^n$-a.e. uniquely determined family of Borel probability measures $\lbrace{\tilde{\mu}_{t,y}^n}\rbrace_{y \in B_n} \subset \mathscr{P}(B)$ such that $ \tilde{\mu}_{t,y}^n(B \setminus T_n^{-1}(y))=0$ for $\tilde{\mu}_t^n $-a.e. $y\in B_n$ and
$$ \int_{B} f(x) \, \mu_t(dx)= \int_{B_n} \Bigl( \int_{T_n^{-1}(y) } f(x) \, \tilde\mu_{t,y}^n(dx)\Bigr) \, \tilde\mu_t^n(dy) $$
for every Borel map $f:B \rightarrow [0,+\infty]$ with $T_n$ is the map introduced in Section \ref{project} and satisfying the properties in Lemma \ref{weaklyconv}.
\item Recall \eqref{tildeeta} and \eqref{imagemeasure}. We have
\be \label{measuren}
\tilde\mu^n_t =(\Xi_t)_{{\sharp}} \tilde\eta^n
\ee
for all $t\in I$. Moreover,
$\tilde\eta^n$ is concentrated on the set of pairs $(x,\gamma)$ such that $\gamma  \in  AC^1_{\it loc}(I;B_n)$ is a solution of the $\dot{\gamma}(t)=\tilde{v}^n(t,\gamma(t)), \,\gamma(t_0)=x \in B_n$ for a.e. $t\in I$ where the vector field  $\tilde{v}^n$ is the projected one defined as follows
\begin{equation}\label{tildevn}
\tilde{v}^n(t,y):=\int_{{T_n}^{-1}(y)} T_n \circ v(t,x)\, \tilde{\mu}_{t,y}^n(dx) \quad \text{for $ \tilde{\mu}_t^n$ a.e. $t\in I.$}
\end{equation}
We obtain \eqref{tildevn} by arguing analogously as for \eqref{vd} in the proof of Lemma \ref{chap.3.proj.arg} above.
\item We have also
\[ T_n \circ h(t,y)= \int_{T_n^{-1}(y)} T_n \circ h(t,T_n(x)) \, \tilde{\mu}_{t,y}^n(dx), \quad  \forall y \in B_n. \]
The above identity is true by the support property of $\tilde{\mu}_{t,y}^n $.
\end{enumerate}
To prove (\ref{bdd}), we start with
\begin{equation}\label{decom}
\begin{aligned}
&\int_{\mathfrak X} \left\lVert \gamma(t)-x-\int_{t_0}^{t}h(\tau,\gamma(\tau)) \, d\tau \right\rVert_{*} \, \,\tilde\eta^n(dx,d\gamma)\\ &
\leq\int_{\mathfrak X} \left\lVert \gamma(t)-x-\int_{t_0}^{t}T_n \circ h(\tau,\gamma(\tau)) \, d\tau \right\rVert_{*} \, \,\tilde\eta^n(dx,d\gamma)
\\ & +\int_{\mathfrak X} \left\lVert \int_{t_0}^{t}\Big[ T_n \circ h(\tau,\gamma(\tau)) - h(\tau,\gamma(\tau)) \Big] \, d\tau \right\rVert_{*} \, \,\tilde\eta^n(dx,d\gamma).
  \end{aligned}
  \end{equation}
    Gathering (1), (2) and (3),  \eqref{measuren}, Lemma \ref{weaklyconv}, then using the disintegration Theorem \ref{disint} and  the Lebesgue dominated convergence theorem, the second  line in (\ref{decom}) gives
\[\begin{aligned}
&\int_{\mathfrak X} \left\lVert \gamma(t)-x-\int_{t_0}^{t}T_n \circ h(\tau,\gamma(\tau)) \, d\tau \right\rVert_{*} \, \,\tilde\eta^n(dx,d\gamma)
\\&\leq
\int_{\mathfrak X_n} \left\lVert \int_{[t_0,t]} \Big[ \tilde{v}^n(\tau,\gamma(\tau))- T_n \circ h(\tau,\gamma(\tau)) \Big] \, d\tau \right\rVert_{*} \, \,\tilde{\eta}^n(dx,d\gamma)
\\&\leq
 \int_{[t_0,t]}\int_{\mathfrak X_n} \left\lVert  \tilde{v}^n(\tau,\gamma(\tau))- T_n \circ h(\tau,\gamma(\tau))  \right\rVert_{*} \, \,\tilde{\eta}^n(dx,d\gamma) \, d\tau
 \\&=
 \int_{[t_0,t]}\int_{B_n} \left\lVert  \tilde{v}^n(\tau,y)- T_n \circ h(\tau,y)  \right\rVert_{*} \, \,\tilde{\mu}^n_\tau(dy) \, d\tau
  \\&=
 \int_{[t_0,t]}\int_{B_n} \left\lVert\int_{T_n^{-1}(y)} T_n \circ v(\tau,x) \, \tilde{\mu}_{\tau,y}^n(dx)  - \int_{T_n^{-1}(y)} T_n \circ h(\tau,T_n(x)) \, \tilde{\mu}_{\tau,y}^n(dx) \right\rVert_{*} \, \tilde{\mu}^n_\tau(dy) \, d\tau
 \\&\leq \int_{[t_0,t]}\int_{B} \left\lVert T_n \circ v(\tau,x)  -  T_n \circ h(\tau,T_n(x)) \right\rVert_{*} \, {\mu}_\tau(dx) \, d\tau
 \\&\leq  \int_{[t_0,t]}\int_{B} \left\lVert  v(\tau,x)  -  h(\tau,T_n(x)) \right\rVert_{*} \, {\mu}_\tau(dx) \, d\tau
 \\&\leq  \Big[  \int_{[t_0,t]}\int_{B} \left\lVert  v(\tau,x)  -  h(\tau,x) \right\rVert_{*} \, {\mu}_\tau(dx) \, d\tau + \int_{[t_0,t]}\int_{B} \left\lVert  h(\tau,T_nx)  -  h(\tau,x) \right\rVert_{*} \, {\mu}_\tau(dx) \, d\tau \Big]
 \\&=    \int_{[t_0,t]}\int_{B} \left\lVert  v(\tau,x)  -  h(\tau,x) \right\rVert_* \, {\mu}_\tau(dx) \, d\tau + \eps_1(n),
\end{aligned} \]
with $\eps_1(n) \longrightarrow 0$ as $n\rightarrow +\infty$.
Again by the  Lebesgue dominated convergence theorem,
 the third line of (\ref{decom}) gives
\[\begin{aligned}
&\int_{\mathfrak X} \left\lVert \int_{t_0}^{t}\Big[ T_n \circ h(\tau,\gamma(\tau)) - h(\tau,\gamma(\tau)) \Big] \, d\tau \right\rVert_{*} \, \,\tilde\eta^n(dx,d\gamma) \\& \leq\int_{[t_0,t]}  \int_{\mathfrak X} \left\lVert  T_n \circ h(\tau,\gamma(\tau)) - h(\tau,\gamma(\tau)) \right\rVert_{*} \, \,\tilde\eta^n(dx,d\gamma) \, d\tau
\\& =\int_{[t_0,t]}  \int_{B_n} \left\lVert  T_n \circ h(\tau,x) - h(\tau,x) \right\rVert_{*} \, \tilde\mu_\tau^n(dx) \, d\tau
\\& =\int_{[t_0,t]}  \int_{B_n} \left\lVert  T_n \circ h(\tau,T_nx) - h(\tau,T_nx) \right\rVert_{*} \, \mu_\tau(dx) \, d\tau
\\& \leq \eps_2(n),
 \end{aligned}\]
 with $\eps_2(n) \longrightarrow 0$ as $n\rightarrow +\infty$. Combining the bounds on the second and third lines in (\ref{decom}) and  the above calculations, we conclude
\[\begin{aligned}
\int_{\mathfrak X} \left\lVert \gamma(t)-x-\int_{0}^{t}h(\tau,\gamma(\tau)) \, d\tau \right\rVert_{*} \, \tilde\eta^n(dx,d\gamma) \leq &
 \int_{[t_0,t]}\int_{B} \left\lVert  v(\tau,x)  -  h(\tau,x) \right\rVert_* \, {\mu}_\tau(dx) \, d\tau \\ &+ \eps_1(n)+\eps_2(n).
\end{aligned} \]
Letting $n \rightarrow +\infty$  in the above equality, we  get (\ref{bdd}).
We have then
\[\begin{aligned}
&\int_{\mathfrak X} \left\lVert \gamma(t)-x-\int_{t_0}^{t}v(\tau,\gamma(\tau)) \, d\tau \right\rVert_{*} \, \,\eta(dx,d\gamma)
\\ &\leq
\int_{\mathfrak X} \left\lVert \gamma(t)-x-\int_{t_0}^{t}h(\tau,\gamma(\tau)) \, d\tau \right\rVert_{*} \, \eta(dx,d\gamma)+
\int_{\mathfrak X} \left\lVert \int_{t_0}^{t}\Big[ v(\tau,\gamma(\tau))- h(\tau,\gamma(\tau)) \Big] \, d\tau \right\rVert_{*}\, \eta(dx,d\gamma)
\\ &\leq 2\, \int_{[t_0,t]} \int_{B} \Vert v(\tau,x) -h(\tau,x)\Vert_* \, \mu_\tau(dx) \, d\tau
\\& \leq 2 \Big[ \int_{[t_0,t]}\int_{ B} \|T_n\circ v(\tau,x)- v(\tau,x)\|_* \, \mu_\tau(dx) \, d\tau+  \int_{[t_0,t]}\int_{ B} \|T_n\circ v(\tau,x)-h(\tau,x)\|_* \, \mu_\tau(dx) \, d\tau \Big]\,.
\end{aligned}\]
First, since we have \eqref{chap3.s1.eq3}, then by the  Lebesgue dominated convergence theorem
\[ \int_{[t_0,t]}\int_{ B} \|T_n\circ v(\tau,x)-v(\tau,x)\|_* \, \mu_\tau(dx) \, d\tau \underset{n\rightarrow+\infty}{\longrightarrow}0. \]
Remember that $\Vert\cdot\Vert_*\leq \Vert\cdot\Vert$. So, it remains to seek a sequence of continuous bounded functions $h_n: [t_0,t]\times B\rightarrow B$ such that
\[\lim_{n\rightarrow +\infty}\int_{[t_0,t]}\int_{ B} \|T_n \circ v(\tau,x)-h_n(\tau,x)\|_B \, \,\mu_\tau(dx) \, d\tau=0.  \]
To this end, set for $a,b \in [t_0,t]$ with and for every Borel set $K$ in $B$
\[\tilde{\nu}([a,b]\times K):=\int_{[a,b]} \mu_\tau(K) \, d\tau.  \]
We have
\[ \int_{[t_0,t] \times B} \|T_n \circ v(\tau,x)\|_B \, \tilde\nu(d\tau,dx)\leq   \int_{[t_0,t]}\int_{ B}\| v(\tau,x)\|_B \, \mu_\tau(dx) \, d\tau <+\infty.\]
This implies $T_n \circ v \in L^1([t_0,t] \times B,\tilde\nu;B_n)$. Remark that  $B_n$ can be  identified with $\RR^{n}$. Then, applying  Lemma \ref{density}, one obtains a sequence of continuous bounded functions $(h_n)_{n\in \NN}$ from $[t_0,t]\times B $ to $B_n\subset B$ such that
\[\|T_n\circ v-h_n\|_{L^1([t_0,t] \times B,\tilde\nu;B_n)} \underset{n \rightarrow +\infty}{\longrightarrow 0}.\]
Hence, we have
\[\begin{aligned} \int_{[t_0,t]}\int_{ B} \|T_n\circ v(\tau,x)-h_n(\tau,x)\|_B \,\, \mu_\tau(dx) \, d\tau=&\int_{[t_0,t]\times B} \|T_n\circ v(\tau,x)-h_n(\tau,x)\|_B \, \,\tilde\nu(d\tau,dx) \\ &\underset{n\rightarrow \infty}{\longrightarrow}0.
\end{aligned}\]
We conclude
\[ \int_{\mathfrak X} \left\lVert \gamma(t)-x-\int_{t_0}^{t}v(\tau,\gamma(\tau)) \, d\tau \right\rVert_{*} \, \eta(dx,d\gamma) =0. \]
This implies that there exists an $\eta$ null set $\cN$ such that
\[ \gamma(t)=x+\int_{t_0}^{t}v(\tau,\gamma(\tau)) \, d\tau =0, \quad \forall   (x,\gamma) \in \mathfrak X \setminus \cN.\]
Then using a density argument and the continuity of the curves $\gamma$ in $B$ as in the finite-dimensional case, we obtain the concentration property   \ref{chap.3.(i)} in  Proposition \ref{propinfinite}.
\end{proof}

\section{The globalization argument}
\label{chap3.sec.app.proof}

We give the proof of our main results. In particular,  Theorem \ref{chap3.main.thm} is proved in
the subsection below while the applications to ODEs and PDEs are analyzed in Subsection \ref{chap3.sec.PDEana}.

\subsection{Proof of main results}
In order to prove Theorem \ref{chap3.main.thm}, we  rely  on the global superposition principle Proposition \ref{propinfinite}
proved in  Section \ref{chap.3.sec.2} and the measurable projection theorem recalled below (see \cite[Theorem 2.12]{MR1993844}).

\begin{theorem}[Measurable projection theorem]\label{chap.3.projmes}
  Let $(X,\mathscr T)$ be a measurable space and let $(Y,\mathscr B)$ be
 a Polish space with Borel $\sigma$-algebra $\mathscr{B}$. Then for every set $S$ in the product $\sigma$-algebra  $
\mathscr {T}\otimes \mathscr{B}$  the projected set $p(S)$, $p:X\times Y\to X, p(x,y)=x$,  is a universally measurable set of  $X$ relatively to $\mathscr T$.
\end{theorem}

We will also need a measure theoretical  argument provided for instance in \cite[Lemma 4.3 and Lemma C.2]{alc2020} and in the PhD thesis of C.~Rouffort \cite[Lemma 3.A.1]{rouff2018}.

\begin{lemma}[see {\cite[Lemma 3.A.1]{rouff2018}}]\label{chap.3.mesmap}
Let $(M, d_M)$ be a metric space, let $a, b\in\RR$, $a < b$. Then, for any Borel measurable function
$f : [a, b] \times M \to \RR$ such that for all $u\in M, f(\cdot, u)\in L^1([a, b])$, the mapping given by
\begin{eqnarray}
M &\longrightarrow & \RR \\
x &\longmapsto &\int_{[a,b]} f(s,x) ds
\end{eqnarray}
is Borel measurable.
\end{lemma}

\medskip
Define the set
\begin{equation}\label{chap.3.Fset}
\mathscr{F}_{t_0}=\{(x,\gamma) \in\mathfrak{X} : \gamma \text{ is a  global mild solution of \eqref{chap3.ivpinh} s.t.  } \gamma(t_0)=x \}\,.
\end{equation}

\begin{lemma}
\label{sec.glob.lem.3}
The set $\mathscr{F}_{t_0}$ is a Borel subset of  $(\mathfrak{X}, d)$ satisfying $\eta(\mathscr{F}_{t_0})=1$  where   $\eta$ is the Borel probability measure on $ \mathfrak{X}$ provided by the global superposition principle in Proposition \ref{propinfinite}.
\end{lemma}
\begin{proof}
According to Proposition \ref{propinfinite}, the path measure $\eta$ constructed there concentrates on the set $\mathscr{F}_{t_0}$ of global solutions with specified initial conditions given in \eqref{chap.3.Fset}. More precisely, the concentration property \ref{chap.3.(i)} says that $\mathfrak X\setminus{\mathscr F_{t_0}}$ is a $\eta$-null set. So, to prove the  lemma it is enough to show that $\mathscr F_{t_0}$ is a Borel subset of $\mathfrak X$.  Such a statement follows from Lemma \ref{chap.3.mesmap}.  We have \
\begin{eqnarray}
\mathscr{F}_{t_0}&=&\{(x,\gamma) \in\mathfrak{X} : v(\cdot,\gamma(\cdot))\in L^1_{loc}(I) \text{ and }  \gamma(t)=x+\int_{t_0}^t v(s,\gamma(s)) ds, \forall t\in I\}\\
&=&\bigcap_{t_j\in\mathbb Q\cap I} \{(x,\gamma) \in\mathfrak{X} : v(\cdot,\gamma(\cdot))\in L^1_{loc}(I)  \text{ and }  \gamma(t_j)=x+\int_{t_0}^{t_j} v(s,\gamma(s)) ds\}
\end{eqnarray}
On the other hand, for each $j\in\NN$
\begin{eqnarray}
\mathscr{F}_{t_0}(j)
&:=&\{(x,\gamma) : v(\cdot,\gamma(\cdot))\in L^1_{loc}(I)  \text{ and }  \gamma(t_j)=x+\int_{t_0}^{t_j} v(s,\gamma(s)) ds\}\\
&=&\bigcap_{k\in \NN} \{(x,\gamma) : v(\cdot,\gamma(\cdot))\in L^1_{loc}(I)  \text{ and }
\langle\gamma(t_j)-x-\int_{t_0}^{t_j} v(s,\gamma(s)) ds, e_k\rangle =0\} \\
&=:&\bigcap_{k\in \NN} E_{j,k}\,,
\end{eqnarray}
where $\{e_k\}_{k\in\NN}$ is the elements of the biorthogonal system in $B$ (see Definition \ref{biorthogonal_system}).  Hence, it is enough to show $E_{j,k}$ are Borel sets.  Let $\mathfrak{L} (I, B)$ denote  the set of curves $\gamma \in\mathscr{C}(I;B)$ such that
$v(\cdot,\gamma(\cdot))\in L^1_{loc}(I;B)$. Taking the functions
\begin{eqnarray*}
\Lambda_T: \mathscr{C}(I;B) &\longrightarrow &   \overline{\RR}  \\
\gamma &\longmapsto & \int_{[-T,T]\cap I} \Vert v(s,\gamma(s))\Vert \, ds\\ && =\lim_{N\to+\infty}
\int_{[-T,T]\cap I} \min(N, \Vert v(s,\gamma(s))\Vert) \, ds,
\end{eqnarray*}
we prove using Lemma \ref{chap.3.mesmap} with $M=\mathscr{C}(I;B)$ and the monotone convergence theorem that $\Lambda_T$ are Borel measurable for all $T\in\NN$.  Hence, we conclude that $\mathfrak{L} (I, B)=\cap_{T\in\NN} \Lambda_T^{-1}(\RR)$ is a Borel subset of $\mathscr{C}(I;B)$.
In particular, Borel sets of $(\mathfrak{L} (I, B),d_0)$ equipped with the induced metric $d_{0}$ in \eqref{chap.3.norm-m} coincide with
Borel sets of $\mathscr{C}(I;B)$ which are in $\mathfrak{L} (I, B)$.  Now, using again Lemma \ref{chap.3.mesmap}
with $M=(\mathfrak{L} (I, B),d_0)$, we show that the map
\begin{eqnarray*}
\psi_1: B\times \mathfrak{L} (I, B) &\longrightarrow & \RR \\
(x,\gamma) &\longmapsto & \int_{t_0}^{t_j} \langle v(s,\gamma(s)), e_k\rangle \, ds
\end{eqnarray*}
is Borel measurable. Moreover, the map
\begin{eqnarray*}
\psi_2: B\times \mathfrak{L} (I, B) &\longrightarrow & \RR \\
(x,\gamma) &\longmapsto & \langle \gamma(t_j)-x, e_k\rangle
\end{eqnarray*}
is continuous. Hence, we get
\begin{equation*}
E_{j,k}= (\psi_1-\psi_2)^{-1} (\{0\})
\end{equation*}
is Borel.

\end{proof}

Define the set
\begin{equation}
\mathscr{G}_{t_0}=\{x\in B : \exists \gamma \text{ a  global mild solution of \eqref{chap3.ivpinh} s.t. } \gamma(t_0)=x  \}\,.
\end{equation}

\begin{lemma}
\label{sec.glob.lem.4}
The set $\mathscr{G}_{t_0}$ is a universally measurable subset of $(B,\Vert\cdot\Vert)$.
\end{lemma}
\begin{proof}
Take the projection map $p: B\times \mathscr{C}( I; B)\to B$, $p(x,\gamma)=x$. Recall that according to Lemma \ref{sec.glob.lem.1},
$\mathscr{C}( I; B)$ endowed with the metric $d_0$ of  the compact-open topology is a Polish space. Then using Lemma \ref{sec.glob.lem.3} and the measurable projection Theorem \ref{chap.3.projmes}, we obtain that
$$
p(\mathscr F_{t_0})= \mathscr G_{t_0}
$$
is a universally measurable set of $(B,\Vert\cdot\Vert)$.
\end{proof}

\medskip\noindent
{\bf Proof of Theorem  \ref{chap3.main.thm}:}
The global superposition principle, Proposition \ref{propinfinite}, yields the existence of a probability measure
$\eta\in \mathscr{P}(\mathfrak{X})$ such that $\mu_t =(\Xi_t)_{{\sharp}} \eta$, for all $t\in I$ where  $\Xi_t$ is the evaluation map
$$(x,\gamma)\in\mathfrak{X}\mapsto \gamma(t)\in B.$$
Thanks to Lemma \ref{sec.glob.lem.4}, we know that the set $\mathscr{G}_{t_0}$ is $\mu_{t_0}$-measurable since it is universally measurable.
Hence, Proposition \ref{propinfinite} and Lemma \ref{sec.glob.lem.3} imply
\begin{equation}
\label{eq_mu_eta}
\mu_{t_0}(\mathscr{G}_{t_0})= (\Xi_{t_0})_{\sharp}\eta(\mathscr{G}_{t_0})= \eta\big(\Xi_{t_0}^{-1}(\mathscr{G}_{t_0})\big)\geq \eta(\mathscr{F}_{t_0})=1\,.
\end{equation}
The last inequality is a consequence of  the inclusion $\mathscr{F}_{t_0}\subset \Xi_{t_0}^{-1}(\mathscr{G}_{t_0})$.   $\hfill\square$

\medskip
 Before proceeding with the proof of Theorem \ref{thm_measurable_flow}, we note the following measure-theoretic result, which is proved for instance in \cite[Theorem 3.9]{MR0226684}.

\begin{lemma}
\label{inj_mes}
 Let $X_{1}$ and $X_{2}$ be two complete separable metric spaces with $E_{1} \subset X_{1}$ and  $E_{2} \subset X_{2}$. Suppose that $E_{1}$ is a Borel set. Let $\psi$ be a measurable one-to-one map of $E_{1}$ into $X_{2}$ such that $\psi(E_{1})=E_{2}$. Then $E_{2}$ is a Borel set of $X_{2}$.
\end{lemma}

 Let us now show how Theorem \ref{chap3.main.thm} and Lemma \ref{inj_mes} imply Theorem \ref{thm_measurable_flow}.

\medskip\noindent
{\bf Proof of Theorem  \ref{thm_measurable_flow}:}
  Let $t_0\in\RR$ be any initial time. Consider $X_1=(\mathfrak{X}, d)$ given by \eqref{mathcalH}-\eqref{d_metric}, $X_2=B$ and $E_{1}=\mathscr{F}_{t_0}$ defined by \eqref{chap.3.Fset}.
  Recall that $(\mathfrak{X}, d)$ is  a (Polish) complete separable metric  space by Lemma \ref{sec.glob.lem.1} and  $\mathscr{F}_{t_0}$ is Borel measurable  by Lemma \ref{sec.glob.lem.3}.
  Let $\psi$  be the (measurable) projection map
$p: \mathscr{F}_{t_0}\subset\mathfrak{X}\to B$, $p(x,\gamma)=x$. Since by assumption, for any $x\in B$ the initial value problem  \eqref{chap3.ivpinh} admits at most one global mild solution, one deduces
that $p$ is a one-to-one map. Hence, according to Lemma \ref{inj_mes}, we have  that $p(\mathscr{F}_{t_0})$ is Borel measurable. Moreover, we have
$$
\mathscr{G}_{t_0}=p(\mathscr{F}_{t_0})\,, \quad \text{ and } \quad \mu_{t_0}(\mathscr{G}_{t_0})
=\eta(\mathscr{F}_{t_0})=1\,.
$$
The latter equality is a consequence of \eqref{eq_mu_eta}. By  Lemma  \ref{inj_mes}, we conclude  that
$p^{-1}:\mathscr{G}_{t_0} \to \mathscr{F}_{t_0}$ is Borel measurable.  Therefore,  the flow map
\begin{eqnarray*}
\phi_{t_0}^t : \mathscr{G}_{t_0} &\overset{p^{-1}}{\longrightarrow}  \mathscr{F}_{t_0} & \overset{\Xi_{t}}{\longrightarrow} B\\
x & \longmapsto  (x,\gamma) & \longmapsto  \gamma(t)\,,
\end{eqnarray*}
is well-defined and Borel measurable by composition. Now, we check  that  $\phi_{t_0}^t$ is a measurable flow satisfying Definition \ref{def.mes.flow}. By construction, the properties $\phi_{t_0}^{t_0}=\mathrm{Id}$ and $t\mapsto \phi_{t_0}^{t}(x)=\gamma(t)\in  \mathscr{C}(\RR;B)$ hold true. Thanks to the uniqueness assumption, we see that for all $t,t_0\in\RR$,
$$
\mathscr{G}_{t}=\phi_{t_0}^t(\mathscr{G}_{t_0})\,.
$$
In particular, for any $x\in\mathscr{G}_{t_0}$ let $\gamma_x(\cdot)$  denote  the (unique)  global mild solution of the  initial value problem \eqref{chap3.ivpinh} satisfying the initial condition $\gamma_x(t_0)=x$, then we check
$$
\phi_{s}^t\circ\phi_{t_0}^s(x)= \phi_{s}^t(\gamma_x(s))=\gamma_x(t)=\phi_{t_0}^t(x)\,,
$$
since $y=\gamma_x(s)\in \mathscr{G}_{s}$ and the curve $\gamma(\cdot)=\gamma_x(\cdot)$ is the unique global mild solution of \eqref{chap3.ivpinh} satisfying $\gamma(s)=y$. Finally, we have for any bounded measurable function $F: B\to \RR$,
\begin{eqnarray}
\label{mes_rel_1}
\int_{\mathscr{G}_{t_0}} F(\phi_{t_0}^t(x)) \, \mu_{t_0}(dx) &=&
\int_{\mathscr{F}_{t_0}} F(\phi_{t_0}^t(\Xi_{t_0}(x,\gamma))) \, \eta(dx,d\gamma) \\ \label{mes_rel_2}
&=&\int_{\mathscr{F}_{t_0}} F(\gamma(t)) \, \eta(dx,d\gamma)\\ \label{mes_rel_3}
&=& \int_{\mathfrak{X}} F(\Xi_{t}(x,\gamma)) \, \eta(dx,d\gamma)\\ \label{mes_rel_4}
&=& \int_{\mathscr{G}_{t}} F(x) \, \mu_t(dx)\,.
\end{eqnarray}
Recall that $\mu_t$ concentrates on the set $\mathscr{G}_t$ and $\eta$ on the set $\mathscr{F}_{t_0}$. The equalities  \eqref{mes_rel_1} and \eqref{mes_rel_4} are consequence of the superposition principle of Proposition \ref{propinfinite}, while  \eqref{mes_rel_2} follows from the definition of the flow $\phi_{t_0}^t$. This proves that $(\phi_{t_0}^t)_\sharp\mu_{t_0}=\mu_t$ for all $t,t_0\in\RR$. \hfill$\square$

\begin{remark}\label{rem.mes.flow}
We remark furthermore that by composition the map
\begin{align*}
\phi_{t_0}^{\cdot} : \RR\times\mathscr{G}_{t_0} &\overset{\mathrm{Id}\times p^{-1}}{\longrightarrow} \mathscr{F}_{t_0}  \overset{\Xi_{t}}{\longrightarrow}  B\\
(t,x) & \longmapsto   (t,p^{-1}(x)) \longmapsto \gamma(t)\,,
\end{align*}
is Borel measurable.
\end{remark}

\medskip\noindent
{\bf Proof of Theorem  \ref{thm_flow_invariance}:}
We assume that the initial value problem \eqref{chap3.ivpinh} admits a measurable flow $(\phi_{t_0}^t)$ with respect to $(\mu_t)_{t\in \RR}$ as in Definition \ref{def.mes.flow}. We intend to prove that the curve $(\mu_t)_{t\in \RR}$ satisfies the statistical Liouville equation \eqref{chap3.le}. By Fubini for any
$F \in \mathscr{C}^\infty_{c,cyl}(B)$ and $\chi\in\mathscr{C}^\infty_{c}(\RR)$,
\begin{eqnarray}
\label{eq_int_part_1}
\int_{\RR} \chi'(t) \int_{\mathscr{G}_{t_0}} F(\phi_{t_0}^t(x)) \,\mu_{t_0}(dx) dt &=&
 \int_{\mathscr{G}_{t_0}} \int_{\RR} \chi'(t)  F(\phi_{t_0}^t(x)) \,dt \,\mu_{t_0}(dx).
\end{eqnarray}
For any $x\in \mathscr{G}_{t_0}$, by integration by parts
\begin{eqnarray}
\label{eq_int_part_2}
\int_{\RR} \chi'(t)  F(\phi_{t_0}^t(x)) \,dt &=& -\int_{\RR} \chi(t)  \frac{d}{dt}F(\phi_{t_0}^t(x)) \,dt.
\end{eqnarray}
The above equality holds true because the function $t\mapsto F(\phi_{t_0}^t(x))$ is absolutely  continuous
 admitting a Lebesgue almost everywhere derivative locally integrable for any $x\in\mathscr{G}_{t_0}$. In particular, the curve $t\mapsto \phi_{t_0}^t(x)$ is a global mild solution satisfying
$t\mapsto v(t,  \phi_{t_0}^t(x))\in L^1_{loc}(\RR,dt)$ and for a.e. $t\in\RR$
$$
\frac{d}{dt}F( \phi_{t_0}^t(x))=\langle v(t,\phi_{t_0}^t(x)),\nabla F(\phi_{t_0}^t(x)) \rangle\,,
$$
with
$$
\lvert \langle v(t,\phi_{t_0}^t(x)), \nabla F(\phi_{t_0}^t(x)) \rangle \rvert  \lesssim \lVert
v(t,\phi_{t_0}^t(x))\rVert\in L^1_{loc}(\RR,dt)\,.
$$
 Hence, by \eqref{eq_int_part_1}-\eqref{eq_int_part_2}, we have
\begin{eqnarray*}
\int_{\RR} \chi'(t) \int_{\mathscr{G}_{t_0}} F(\phi_{t_0}^t(x)) \,\mu_{t_0}(dx) dt &=&-
  \int_{\RR} \chi(t) \int_{\mathscr{G}_{t_0}}   \langle v(t,\phi_{t_0}^t(x)), \nabla F(\phi_{t_0}^t(x)) \rangle \, \mu_{t_0}(dx) dt \\
  &=& -
  \int_{\RR} \chi(t) \int_{\mathscr{G}_{t_0}}   \langle  v(t,x), \nabla F(x)\rangle \, \mu_{t}(dx) dt.
\end{eqnarray*}
The latter equality yields the statistical Liouville equation \eqref{chap3.le} in the distributional sense. \hfill$\square$

\subsection{Analysis of ODEs and PDEs}
\label{chap3.sec.PDEana}
We sketch here the proofs of Corollaries \ref{cor.ode} and \ref{chap3.cor.res} which provide respectively two  applications for  ODEs and PDEs.

\medskip
\noindent
\paragraph{\emph{Analysis of ODEs:}}  Recall the initial value problem \eqref{chap3.ode} and assume that the assumptions of Corollary \ref{cor.ode} are satisfied.

{\bf Proof of Corollary \ref{cor.ode}:}
Without loss of generality, we may assume that the symplectic structure $J$  is canonical.  Precisely, the skew-symmetric matrix $J$ satisfying $J^2=-\rm I_{2d}$ is given by
$J=\begin{pmatrix}
0 & \rm I_d \\
- \rm I_d & 0
\end{pmatrix}$.  For any $\varphi\in\mathscr{C}_c^\infty(\RR^{2d})$, we have by integration by parts
\begin{eqnarray*}
\int_{\RR^{2d}} \frac{\partial h}{\partial q} (u)  \frac{\partial \varphi}{\partial p}(u) \;F(h(u)) \, L(du) = \\
\int_{\RR^{2d}}  \varphi(u) \Bigl( -\frac{\partial^2 h}{\partial p\partial q} (u)  \, F(h(u)) -
 \frac{\partial h}{\partial q} (u) \frac{\partial h}{\partial p} (u) \, F'(h(u))\Bigr) \, L(du),
\end{eqnarray*}
and  similarly
\begin{eqnarray*}
- \int_{\RR^{2d}}  \frac{\partial h}{\partial p} (u)  \frac{\partial \varphi}{\partial q}(u) \;F(h(u)) \, L(du) = \\
\int_{\RR^{2d}}  \varphi(u) \Bigl( \frac{\partial^2 h}{\partial q\partial p} (u)  \, F(h(u)) +
\frac{\partial h}{\partial p} (u) \frac{\partial h}{\partial q} (u) \, F'(h(u))\Bigr) \, L(du).
\end{eqnarray*}
Combining the two identities, we obtain
\begin{eqnarray}
\int_{\RR^{2d}} \Bigl( \frac{\partial h}{\partial q} (u)  \frac{\partial \varphi}{\partial p}(u)-  \frac{\partial h}{\partial p} (u)  \frac{\partial \varphi}{\partial q}(u)\Bigr) \;F(h(u)) \, L(du) =0.
\end{eqnarray}
Hence, using the symplectic structure on $\RR^{2d}$ and the Hamiltonian character  of the initial value problem, we prove
\begin{eqnarray}
\int_{\RR^{2d}} \bigl\langle \nabla\varphi, J\nabla h(u)\bigr\rangle \;F(h(u)) \, L(du) =0.
\end{eqnarray}
This shows  that the  measure  $\frac{F(h(\cdot)) dL}{\int F(h(u)) L(du)}$  satisfies the statistical Liouville equation and so we are within the framework of Theorem
\ref{chap3.main.thm}.  The latter grants us  the almost sure existence of global solutions to the ODE \eqref{chap3.ode}.
$\hfill\square$

\bigskip
\paragraph{\emph{Analysis of PDEs:}}
In this paragraph, we provide the proof of  Proposition \ref{chap3.invdisp} and Corollary \ref{chap3.cor.res}. Recall that the
initial value problem \eqref{chap3.ivpH-s}  can be written equivalently in the interaction representation as
\begin{equation*}
\dot\gamma(t)=v(t,\gamma(t)),
\end{equation*}
with a Borel vector field $v:\RR\times H^{-s}\to H^{-s}$ given by
\begin{equation}
\label{chap3.vgaus}
v(t,u)= -i e^{it A} \nabla h_{NL}( e^{-it A} u),\qquad x\in H^{-s}.
\end{equation}
First, we notice the following invariance:
\begin{proposition}
\label{chap3.stagibbs}
Consider the Gibbs measure
\begin{equation}\label{gibbs-gaussian-measure}
d\mu_0= \frac{e^{-h_{NL}} d\nu_0}{\int_{H^{-s}} e^{-h_{NL}}  d\nu_0}\,.
\end{equation}
Then for any $F\in \mathscr{C}_{b,cyl}^\infty(H^{-s})$,
\begin{equation}
\label{chap3.statmu}
\int_{H^{-s}} \langle \nabla F(u),  -i A u -i \nabla h_{NL}(u)\rangle \, \mu_0(du)=0.
\end{equation}
\end{proposition}
\begin{proof}
For $k \in \NN$, let $e_k$ be as in \eqref{e_k_definition} above. Note that two copies of $\{e_k,i e_k, k\in\NN\}$ forms a fundamental strongly total biorthogonal system on $H^{-s}$.
Using  \cite[Theorem 4.11 and (4.38)]{MR4571599}, we have for all $F,G\in \mathscr{C}_{c,cyl}^\infty(H^{-s})$
\begin{equation}
\label{chap3.poiss}
\int_{H^{-s}} \{F,G\}(u) \, \mu_0(du) =\int_{H^{-s}} \langle \nabla F(u),  -i A u -i \nabla h_{NL}(u)\rangle \,G(u)\, \mu_0(du),
\end{equation}
where $\{\cdot,\cdot\}$ refers to the Poisson bracket {(see \cite{MR4571599} for more details)}. Take  a sequence of functions $(G_n)_{n\in\NN}$ in $\mathscr{C}_{c,cyl}^\infty(H^{-s})$ with given fixed basis\footnote{{\it i.e.} $G(u)=
\varphi(\langle u,e_1\rangle,\cdots, \langle u,e_m\rangle; \langle u, i e_1\rangle,\cdots, \langle u, i e_m\rangle )$, for $m$ fixed.}.
Suppose that $G_n \to 1$ pointwisely  with $\partial_j G_n$ are uniformly bounded with respect to $n$ and  $\partial_j G_n\to 0$ pointwise.  Then replacing $G$ by the sequence $G_n$  and letting $n\to \infty$ in \eqref{chap3.poiss}, yields the identity  \eqref{chap3.statmu}.
\end{proof}

\medskip
\noindent
{\bf Proof of Proposition \ref{chap3.invdisp}:}
We let $\{f_k, k \in \NN\}:=\{e_k,ie_k, k \in \NN\}$.
Let $F\in \mathscr{C}_{b,cyl}^\infty(H^{-s})$ be such that
$$
F(u)=\varphi(\langle u,f_1\rangle,\cdots,\langle u,f_m\rangle),
$$
for some $\varphi\in \mathscr{C}_{b}^\infty(\RR^{m})$. Then
\begin{equation}
\frac{d}{dt} \int_{H^{-s}} F(u)\,\mu_t(du)= \sum_{k=1}^m \int_{H^{-s}} \partial_k\varphi(\langle e^{itA} u, f_1\rangle,\dots, \langle  e^{itA} u, f_m\rangle) \, \langle  f_k, iA e^{itA} u \rangle   \, \mu_0(du).
\end{equation}
On the other hand, we check
  \begin{eqnarray*}
 \int_{H^{-s}} \langle v(t,u), \nabla F(u)\rangle \, \mu_t(du) &=& \\ && \hspace{-.6in}
 \sum_{k=1}^m \int_{H^{-s}} \partial_k\varphi(\langle  e^{itA} u, f_1\rangle,\dots, \langle  e^{itA} u, f_m\rangle)
  \, \langle f_k,-i e^{itA} \nabla h_{NL}(u)  \rangle   \, \mu_0(du).
\end{eqnarray*}
 Hence, Proposition \ref{chap3.stagibbs} yields
  \begin{eqnarray*}
 \int_{H^{-s}} \langle v(t,u), \nabla F(u)\rangle \, \mu_t(du) -\frac{d}{dt} \int_{H^{-s}} F(u)\,\mu_t(du) &=&
 \int_{H^{-s}} \langle \nabla \tilde F(u),  -i A u -i \nabla h_{NL}(u)\rangle \, \mu_0(du)\\
  &=& 0,
\end{eqnarray*}
for $\tilde F\in\mathscr{C}_{b,cyl}^\infty(H^{-s})$ given by
$$
\tilde F(u)=\varphi(\langle  e^{it\lambda_1} u, f_1\rangle,\dots, \langle  e^{it\lambda_m} u, f_m\rangle),
$$
and satisfying
$$
\nabla \tilde F(u)=\sum_{k=1}^m  \partial_k\varphi(\langle  e^{it\lambda_1} u, f_1\rangle,\dots, \langle  e^{it\lambda_m} u, f_m\rangle) \,
e^{-it\lambda_k} f_k\,.
$$
$\hfill\square$

\medskip
\noindent
{\bf Proof of Corollary \ref{chap3.cor.res}:} \\
In this framework: $B=H^{-s}$ is a separable dual Banach space with predual $E=H^s$.  The vector field $v:\RR\times H^{-s}\to H^{-s}$ given by \eqref{chap3.vgaus} is Borel measurable. The Gaussian measure $\nu_0$ and the Gibbs measures $\mu_t$ are well-defined Borel probabilities on $B$.  Then, one checks thanks to Proposition \ref{chap3.invdisp} that   $(\mu_t)_{t\in\RR}$ is a narrowly continuous curve in $\mathscr{P}(B)$ satisfying the condition \eqref{chap3.s1.eq3} and the statistical Liouville equation \eqref{chap3.le}. Hence, applying Theorem \ref{chap3.main.thm}, we obtain the $\mu_0$-almost sure existence of global mild solutions for the initial value  problem \eqref{chap3.ivpH-s}.   Taking now  into account the expression \eqref{gibbs-gaussian-measure} of the measure $\mu_0$,  one deduces the $\nu_0$-almost sure existence of global solutions  as  stated in Corollary \ref{chap3.cor.res}.
$\hfill\square$

\bigskip
\appendix
\addcontentsline{toc}{chapter}{Appendices}
\addtocontents{toc}{\protect\setcounter{tocdepth}{-1}}

\section{Local integrability condition}
\label{sec.loc.intg}
\medskip
We show below the equivalence between the local integrability condition \eqref{locinteg}  and \eqref{chap3.s1.eq3}.
\begin{lemma}
\label{equilocint}
Let  $v: I \times B \rightarrow B$ be a Borel vector field and  let $(\mu_t)_{t\in I}$ be a narrowly continuous curve in $\mathscr{P}(B)$. The statements  \eqref{locinteg} and
\eqref{chap3.s1.eq3} are equivalent.
\end{lemma}
\begin{proof}
For simplicity take $I=\RR_+$. The proof in the general case is analogous. Suppose that \eqref{locinteg} is true and let for $t\in\RR_+$,
$$
f(t)=\int_{B}\Vert v(t,u) \Vert  \,\mu_t(du)\,.
$$
If $f\in L^1(\RR_+,dt)$ then one can take $\omega(t)=1$ for all $t\in \RR_+$ and \eqref{chap3.s1.eq3} is then satisfied. \\
Suppose now that $f\notin L^1(\RR_+,dt)$. Let $M>0$  be given. Then, there exists a non-negative increasing sequence
$(a_n)_{n\in\NN}$, with $a_0=0$, such that for all $n\geq 1$,
\begin{equation}
\frac M 2 \leq \int_{a_{n-1}}^{a_n} f(t) \, dt \leq M.
\end{equation}
We claim that the sequence $(a_n)_{n\in\NN}$ is unbounded (i.e.~$\lim_n a_n=+\infty$). If not, we have for all $k\geq 1$,
\begin{equation}
\label{inq.eq.cond}
k \frac M 2\leq \sum_{n=1}^k \int_{a_{n-1}}^{a_n} f(t) \, dt \leq
\int_{0}^{\sup a_n} f(t) dt<+\infty\,.
\end{equation}
This yields a contradiction as the left hand side of \eqref{inq.eq.cond} tends to $+\infty$ as $k\to\infty$. Hence, define the function $\omega: \RR_+\to \RR_+^*$ such that $\omega(0)=1$, $\omega(a_n)=(n+1)^2$  and
interpolate linearly between these points. Then $\omega$ is a positive  non-decreasing (continuous) function satisfying
$$
0\leq \int_{0}^{+\infty} f(t) \, \frac{dt}{\omega(t)}= \sum_{n=1}^\infty \int_{a_{n-1}}^{a_n} f(t) \, \frac{dt}{\omega(t)}\leq  M \sum_{n=1}^\infty \frac{1}{n^2} <\infty.
$$
The opposite implication  $\eqref{chap3.s1.eq3}\Rightarrow \eqref{locinteg}$ is trivial.
\end{proof}

\section{Results in  finite dimensions}
\label{chap3.sec.appA}
In this appendix we recall some useful results in finite dimensions.
Here, we consider $B=\RR^d$ and $I$ is a closed unbounded interval either equal to $\RR$ or half-closed. Take $t_0$ any initial time if $I=\RR$ or $t_0$ is its end-point if $I$ is half-closed.

\medskip\noindent
The following Cauchy-Lipschitz type theorem holds true.
\begin{theorem}[{\bf Cauchy--Lipschitz theorem} ({\cite[Lemma 8.1.4]{AmbrosioLuigi2005GFIM}})]\label{lemma2.1}
Let  $v: I\times \RR^d\to \RR^d$ be a Borel vector field satisfying the local Lipschitz Assumption \ref{lipass}.  Then, for every $x \in \RR^d$, the initial value problem (\ref{chap3.ivpinh}) admits a unique maximal solution in a relatively open interval $I(x)\subseteq I$ with  $t_0 \in I(x)$.
\end{theorem}

 We recall below the characteristics representation formula for solution of statistical Liouville equations in  finite dimensions with locally Lipschitz vector fields  (see  \cite[Proposition 8.1.8]{AmbrosioLuigi2005GFIM}).

\begin{lemma}[{\bf Characteristics representation formula}]\label{lemma2.2}
Let  $v: I\times \RR^d\to \RR^d$ be a Borel vector field satisfying the local Lipschitz Assumption \ref{lipass}.
Let  $(\mu_t)_{t\in I}$ be a narrowly continuous curve in $\mathscr{P}(\RR^d)$ satisfying (\ref{locinteg})  and
 the statistical Liouville equation (\ref{chap3.le}).  Then for all $T>0$ and  for $\mu_{t_0}$-a.e $x\in \RR^d$,
 the  initial value problem (\ref{chap3.ivpinh}) admits a  solution $\Phi_t(x) \equiv \gamma_x(t)$ over $I\cap[-T,T]$ and
\begin{equation}
\label{chap3.appA.eq.flow}
\mu_t= (\Phi_t)_{\sharp}\,\mu_{t_0}, \quad \forall t\in I\cap[-T,T].
\end{equation}
\end{lemma}

We recall the following regularization argument for solutions of Liouville equations in finite dimensions (see  \cite[Lemma 8.1.9]{AmbrosioLuigi2005GFIM}).

\begin{lemma}[{\bf Approximation by regular curves}]\label{lemma2.6}
Let $v:I\times \RR^d\to\RR^d$ be a Borel vector field and $(\mu_t)_{t\in I}$ be a narrowly continuous curve in $\mathscr{P}(\RR^d)$ satisfying (\ref{locinteg})  and  the statistical Liouville equation (\ref{chap3.le}).
Let $ \rho \in \mathscr C^{\infty}(\RR^d)$ strictly positive such that $\int_{\RR^d} \rho(x) \, dx =1$. Set $\rho_\eps(x)=\frac{1}{\eps^d} \, \rho(\frac{x}{\eps})$  and
\be \label{approximation}\mu_t^{\eps}:=\mu_t\ast \rho_\eps,\qquad  v^{\eps}_t:= \frac{ (v_t \mu_t) \ast \rho_\eps}{ \mu_t^{\eps}}. \ee
Then $(\mu_t^{\eps})_{t\in I}$ is a narrowly continuous solution of the Liouville equation \eqref{chap3.le} with respect to the vector field $v^\eps$.  Moreover,  $v^{\eps}_t$  satisfies the local Lipschitz Assumption \ref{lipass} and the uniform integrability bound
\begin{equation}\label{cdt1}
\int_{\RR^d} \Vert v_t^{\eps}(x)  \Vert_{\RR^d}  \,\mu_t^{\eps}(dx)\leq \int_{\RR^d} \Vert v_t(x)  \Vert_{\RR^d}  \,\mu_t(dx),\quad \forall t\in I.
\end{equation}
 \end{lemma}

 \section{Radon spaces, Tightness}
 \label{chap3.sec.appB}
\noindent \underline{\it Radon spaces:} Let $X$ be a Hausdorff topological space. A Borel measure $\mu$ is a Radon measure if it is locally finite and inner regular. A topological space is called a Radon space if every finite Borel measure is a Radon measure (see e.g.  \cite[chapter 2]{MR0426084}).  {In particular,  separable complete metric spaces  and Suslin spaces \cite[Theorem 10]{MR0426084} are Radon spaces.}\\
\noindent \underline{\it Tightness: } Let $X$ be a separable metric space and $\mathscr{P}(X)$ the set of Borel probability measures on $X$.
A set $\cK \subset \mathscr{P}(X)$ is tight if
\[\forall \eps>0, \ \exists K_\eps \ \text{ compact in $X$ such that } \mu(X\setminus K_\eps) \leq \eps, \ \forall \mu \in \cK.\]
\noindent Prokhorov's theorem  \cite[Theorem 5.1.3]{AmbrosioLuigi2005GFIM} says that any tight set $\cK\subset \mathscr{P}(X)$ is relatively (sequentially) compact in $\mathscr{P}(X)$ in the narrow topology. A useful characterization is given below.
\begin{lemma}[{\cite[Remark 5.1.5]{AmbrosioLuigi2005GFIM}}]\label{lemma2.5}
A set  $\cK \subset \mathscr{P}(X)$  is tight if and only if there exists a function $\varphi :X\longrightarrow [0,+\infty]$ whose sublevel sets $ \lbrace x \in X; \varphi(x)\leq c\rbrace $ are relatively compact in $X$ for all $c \geq 0$ and which satisfies \[\begin{aligned}\sup_{\mu \in \cK} \int_{X} \varphi(x) \,\mu(dx)<+\infty\end{aligned}.\]
\end{lemma}
\begin{lemma}[{\cite[page 108]{AmbrosioLuigi2005GFIM}}]\label{cam} Let $(X,d)$ be a (metric, separable) Radon space. Then every narrowly converging sequence $(\mu_n)_n$ is
tight.
\end{lemma}
\begin{lemma}[{\cite[Lemma 5.2.2]{AmbrosioLuigi2005GFIM} }]\label{lemma2.4}
Let $X$, $X_1$ and $X_2$ be separable spaces and let $r^i: X\longrightarrow X_i$ be continuous maps such that the product map $r:=r^1 \times r^2: X\longrightarrow X_1 \times X_2 $ is proper. Let $\cK \subset \mathscr{P}(X)$ be such that $\cK_i:=(r^i)_{\sharp} {\cK}$ is tight in $\mathscr{P}(X_i)$ for $i=1,2$. Then, $\cK$ is tight in $\mathscr{P}(X)$.
\end{lemma}

\section{Equi-integrability and Dunford-Pettis theorem}
Let $(X,\nu)$ be a measurable space. Recall that the dual of  $L^1\equiv L^1(X,\nu)$ is $L^{\infty} \equiv L^{\infty}(X,\nu)$. More precisely, we have $(L^1)^*=L^\infty$.  The weak topology $ \sigma(L^1,L^\infty)$ on $L^1$ is the coarsest topology associated to the collection of linear functions  $(\varphi_f)_{f\in L^\infty}$ where
\[\begin{array}{rcl}
\varphi_f \, : L^1  & \longrightarrow & \RR \\
g & \longmapsto & \displaystyle (\varphi_f)g=\langle f,g \rangle_{L^1, L^\infty}= \int_{X} f(x)g(x)\, \nu(dx).
\end{array}\]
\noindent
The Dunford-Pettis theorem is a standard criterion that allows us to deduce  the relative compactness of sets in $L^1$ with respect to the weak topology $\sigma(L^1,L^\infty)$.\\
  Let $(X,\Sigma)$ be a measurable space and let $\mu$ be a finite measure on $(X,\Sigma)$. We say that a family $\cF \subset L^1(X,\mu)$ is equi-integrable if for any $\eps>0$ there exists $\delta>0$ such that:
 \[ \forall K \in \Sigma,  \ \mu (K)<\delta \Rightarrow \sup_{f \in \cF}  \int_{K} \big|f \big| \,d\mu <\eps. \]
 \noindent A characterization of equi-integrability is given below.
 \begin{lemma}[{\cite[Proposition 1.27]{MR1857292}}]\label{equivalence}
 A bounded set $\cK$ in $L^1(X,\mu)$ is equi-integrable if and only if
 \[\cK \subset  \lbrace  f \in L^1(X,\mu): \quad \int_{X} \theta(|f|) \, d\mu \leq 1 \rbrace , \]
 for some nondecreasing convex continuous function $\theta:\RR_+\longmapsto [0,\infty]$ satisfying $\theta(t)/t\rightarrow\infty$ when $t\rightarrow \infty$ (superlinear) or equivalently if and only if
 \be \label{limm}\lim_{L\rightarrow +\infty} \sup_{f\in \cK} \int_{\lbrace \vert f \vert>L \rbrace  } \vert f \vert  \,d\mu=0.\ee
 \end{lemma}
 \begin{lemma}[Dunford-Pettis {\cite[Theorem 1.38]{MR1857292}}]\label{Dunford-Pettis} A bounded set $\cK$ in  $L^1(X,\mu)$ is relatively sequentially compact for the weak topology $\sigma(L^1,L^\infty)$ if and only if $\cK$ is equi-integrable.
 \end{lemma}

 \section{Compactness argument}\label{chap.3.appx.comp}
We discuss in this paragraph the main compactness argument used throughout the text.  Let $v:I\times B\to B$ be Borel vector field and $(\mu_t)_{t\in I}$ a weakly narrowly continuous curve in $\mathscr{P}(B)$.

\begin{lemma} \label{remarkmeasure}
Assume that $v$ and $(\mu_t)_{t\in I}$ satisfy the  integrability condition  \eqref{chap3.s1.eq3} for some  non-decreasing positive function $\omega: \RR_+\to \RR^*_+$. Then there exists a non-decreasing super-linear  convex function $\theta: \RR_+\rightarrow [0,+\infty]$  such that
\begin{equation}\label{chap.3.est.theta}
\int_{I} \int_{B} \theta(\Vert v(t,u) \Vert) \, \mu_t(du)\, \frac{dt}{\omega(|t|)} \leq 1.
\end{equation}
\end{lemma}
\begin{proof}
For every  $\alpha, \beta \in I$ and for every Borel subset $K$ of $B$, define the measure
\begin{equation}
\label{chap.3.def.nu}
 \nu([\alpha ,\beta] \times K):= \int_{[\alpha,\beta]} \mu_\tau (K) \, \frac{d\tau}{\omega(|\tau|)}\,.
 \end{equation}
We have
\[ \int_{I \times B}\Vert v(t,x) \Vert\, d\nu(t,u)=\int_{I} \int_{\RR^d} \Vert v(t,u) \Vert_{\RR^d}\, \mu_t(du)\, \frac{dt}{\omega(|t|)}<+\infty,   \]
which means $\Vert v(\cdot) \Vert_{\RR^d} \in L^1(I \times B,\nu)$. Since the singleton $ \lbrace \Vert v(\cdot) \Vert_{} \rbrace $  is a compact set, then by the Dunford-Pettis Theorem \ref{Dunford-Pettis},  it is equi-integrable. And thus, by using Lemma \ref{equivalence}, there exists a non-decreasing super-linear convex function $\theta: \RR_+\rightarrow [0,+\infty]$  such that
\[\int_{I} \int_{B} \theta(\Vert v(t,u) \Vert_{}) \, \mu_t(du)\, \frac{dt}{\omega(|t|)} \leq 1.\]
\end{proof}

\begin{remark}\label{chap.3.rem.bfin}
The previous  result applies for finite dimensions  $B=\RR^d$ with any norm on $\RR^d$.
\end{remark}

\noindent We consider the space of continuous functions $\mathscr{C}(I;B)$ endowed with the compact-open topology induced by the
metric $d_{0,*}$ defined in \eqref{chap.3.d0star}.
\begin{lemma}[Compactness]\label{localtight}
Let $\theta: \RR_+\rightarrow [0,+\infty]$   be a non-decreasing super-linear  convex  function.
Define  the function $g: \mathscr{C}(I;B) \to [0,+\infty]$ as
\begin{align*}
g(\gamma):= \begin{cases}
\displaystyle \int_{I} \theta(\Vert \dot{\gamma} \Vert_{*}) \, \frac{dt}{\omega(|t|)}   \qquad &\text{if} \ \gamma(t_0)=0 \ \text{and}  \ \gamma \in AC_{loc}^1(I;B),\\
+\infty \qquad \qquad \qquad   &\text{if} \ \gamma(t_0)\neq 0 \ \text{or} \ \gamma \notin AC_{loc}^1(I;B).
\end{cases}
\end{align*}
Then for all $c\geq 0$, the sublevel set
$$
\cA_c=\lbrace  \gamma \in \mathscr{C}(I;B); \  g(\gamma)\leq c \rbrace
$$
 is relatively compact in the space $(\mathscr C(I; B_w), d_{0,*})$.
\end{lemma}

\begin{proof}
By the general Arzela-Ascoli theorem (see \cite[Theorem 6.1]{MR0464128}), $\cA_c$ is relatively compact in $(\mathscr C(I; B_w), d_{0,*})$ provided that we prove the following claims:
\begin{itemize}
\item  For all $t \in I$, the set $\cA_c(t)=\lbrace  \gamma(t) ; \gamma \in \cA_c \rbrace $ is relatively compact in $B_w$.
\vskip 3mm
\item The set $\cA_c$ is equicontinuous.
\end{itemize}

\underline{$\bullet \, \cA_c(t)$ relatively compact:} \\
 In fact, remark that $\cA_c(t)$ is  bounded. Indeed, by Jensen's inequality
\[\theta( \|\gamma(t) \|_{*})\leq \theta(\int_{t_0}^{t} \|\dot \gamma(s) \|_{*} \, ds) \leq \int_{[t_0,t]} \theta(\|\dot \gamma(s) \|_{*} )\, ds
 \leq c \,\omega(\max(|t_0|,|t|)), \]
 for all $\gamma\in \cA_c$. Since $\theta$ is superlinear, we get $\cA_c(t)$ is bounded in $B_w$. Now, since $t\in I$ is fixed and
 the norm $\Vert\cdot\Vert_*$ induces the weak-* topology in $B$ on bounded sets, it follows that $\cA_c(t)$ is relatively compact in $B_w$.

\underline{$\bullet \, \cA_c$ is equi-continuous:} \\
The case $c=0$ is trivial, so we may assume $c>0$.
Let $t_1 \in I$, we have  to prove:  $\forall \eps>0, \, \exists \delta>0$,  $\forall t \in I$, $\forall \gamma \in \cA_c$
\[|t-t_1|\leq \delta \, \Rightarrow \,  \|\gamma(t)-\gamma(t_1) \|_{*}\leq \eps. \]
Assume $\delta \leq 1$. We have for $\gamma\in \cA_c$,
\[\begin{aligned}
\|\gamma(t)-\gamma(t_1) \|_{*} & =\|\int_{t_1}^{t}\dot \gamma(s)\, ds \|_{*} \leq \int_{[t_1,t]} \|\dot \gamma(s) \|_{*} \,  ds  \\& \leq  \int_{  \lbrace s \in [t_1,t];\|\dot \gamma(s) \|_{*} \leq L \rbrace } L\,ds+\int_{  \lbrace s \in [t_1,t]; \, \|\dot \gamma(s) \|_{*} > L \rbrace } \|\dot \gamma(s) \|_{*}\,ds
\\ & \leq L \vert  t-t_1 \vert +\int_{  \lbrace s \in [t_1,t];\, \|\dot \gamma(s) \|_{*} > L \rbrace } \|\dot \gamma(s) \|_{*} \,ds\\ & \leq L \vert  t-t_1 \vert +\omega(\max|t|,|t_1|) \,\int_{  \lbrace s\in [t_1,t];\|\dot \gamma(s) \|_{*} > L \rbrace } \|\dot \gamma(s) \|_{*} \,\frac{ds}{\omega(|s|)}.
\end{aligned}
\]
Remark that the set $\mathcal K=\lbrace s\in[t_1,t]\mapsto  \|\dot \gamma(s) \|_{*}   , \, \gamma \in \cA_c \rbrace  \subset L^1([t_1,t];\frac{ds}{\omega(|s|)}) $ since the function $\theta$ is superlinear and by assumption
\begin{equation}\label{chap.3.estgthea1}
\int_{[t_1,t]}  \theta( \|\dot \gamma(s) \|_{*} ) \frac{ds}{\omega(|s|)} \leq g(\gamma)\leq c.
\end{equation}
On the other hand, this means that there exists $\tilde\theta=\theta/c$ a  non-decreasing convex superlinear function such that
\begin{equation}\label{chap.3.estgthea2}
\int_{[t_1,t]}  \tilde\theta( \|\dot \gamma(s) \|_{*} ) \frac{ds}{\omega(|s|)} \leq 1.
\end{equation}
Hence, by Lemma \ref{equivalence},  we have that
$\mathcal K$ is equi-integrable and   \[\displaystyle \lim_{L\rightarrow \infty}\sup_{\gamma\in \cA_c}
\int_{  \lbrace s\in[t_1,t];\|\dot \gamma(s) \|_{*} > L \rbrace } \|\dot \gamma(s) \|_{*} \,\frac{ds}{\omega(|s|)}=0.\]
This means for given $\eps>0$, one can choose $L>0$ such that for all $\gamma\in\cA_c$
\[\|\gamma(t)-\gamma(t_1) \|_{*} \leq L \delta + \eps/2.\]
Hence, choosing  $0<\delta\leq \frac{\eps}{2L}$, we show that  $\cA_c$ is equi-continuous at $t_1$.
\end{proof}

\begin{remark}\label{chap.3.rem.fincomp}
The above lemma applies, mutatis mutandis, to finite dimensions with any norm on $\RR^d$.
\end{remark}

 \section{Disintegration theorem}
 \label{chap3.sec.appD}
 Let $E$ and $F$ be Radon separable metric spaces. We say that a measure-valued map $x \in E \rightarrow \mu_x \in \mathscr{P}(E)$ is Borel if $x \in F \rightarrow \mu_x(B)$  is a Borel map for any Borel set $B$ of $E$. We recall below the disintegration theorem (see \cite[Theorem 5.3.1]{AmbrosioLuigi2005GFIM}).
 \begin{theorem}\label{disint}
  Let $E$ and $F$ be Radon separable metric spaces and $\mu \in \mathscr{P}(E)$. Let $\pi:E \rightarrow F$ be a Borel map and $\nu=\pi_{\sharp}\mu \in \mathscr{P}(F)$. Then, there exists a $\nu$-a.e. uniquely determined Borel family of probability measures $ \lbrace \mu_y \rbrace_{y\in F} \subset \mathscr{P}(E)$ such that $\mu_y(E\setminus \pi^{-1}(y))=0$ for $\nu$-a.e. $y \in F$ and
  \be \label{f} \int_{E} f(x) \, \mu(dx)= \int_{F} \bigg(  \int_{\pi^{-1}(y)} f(x) \, \mu_y(dx) \bigg) \,  \nu(dy) \ee
  for every Borel map $f : E \rightarrow [0,+\infty]$.
 \end{theorem}

\section{Density argument}
 Let $(X,\mu)$ be a measurable space. Define the space
\[L^1(X,\mu;\RR^n):=\lbrace f:X\rightarrow \RR^n; \ \int_{X} \|f(x)\|_{\RR^n} \, \mu(dx)<+\infty \rbrace \]
\begin{lemma}\label{density}
Let $X$ be a metric space,  $\mu$ a finite Radon measure on $X$. Then the space of bounded Lipschitz functions $Lip_b(X;\RR^n)$ is dense in $L^1(X,\mu;\RR^n)$. In particular, $\overline{\mathscr{C}_b(X;\RR^n)} = L^1(X,\mu;\RR^n)$.
\end{lemma}
\begin{proof}
By the result in \cite[Corollary C.3, Appendix C]{MR3721874}, we have $ Lip_b(X;\RR)$ is dense in $L^1(X,\mu;\RR) $.
Let $f \in L^1(X,\mu;\RR^n)$, then $f=(f^i)_{i=1}^{n}$. Now for every $1\leq i\leq n$, $f^i \in L^1(X,\mu)$. This implies for all $1\leq i\leq n$ there exists  a sequence of Lipschitz bounded functions $(f^i_k)_{k\in \NN}$ such that
\[\|f^i-f^i_k\|_{L^1(X,\mu)} \underset{k\rightarrow+\infty}{\longrightarrow}0.\]
Let $f_k=(f_k^i)_{i=1}^{n}$ a bounded Lipschitz function i.e. $f_k \in Lip_b(X;\RR^n)$, we have
\[\begin{aligned}
\|f-f_k\|_{L^1(X,\mu;\RR^n)}&=\int_{X} \|f(x)-f_k(x)\|_{\RR^n} \, \mu(dx)\\ &{\lesssim}\int_{X} \sum_{i=1}^{n}|f^i(x)-f^i_k(x)| \, \mu(dx)
\\& = \sum_{i=1}^{n}\|f^i-f^i_k\|_{L^1(X,\mu)} \underset{k\rightarrow+\infty}{\longrightarrow}0.
\end{aligned}\]
And thus $ L^1(X,\mu;\RR^n)=\overline{Lip_b(X;\RR^n)}$.  This implies that
$$L^1(X,\mu;\RR^n)=\overline{Lip_b(X;\RR^n)} \subseteq
\overline{\mathscr{C}_b(X;\RR^n)} \subseteq L^1(X,\mu;\RR^n).$$
\end{proof}

\medskip
\paragraph{\textbf{Acknowledgements:}}
  This research has been funded by the ANR-DFG project  (ANR-22-CE92-0013, DFG PE 3245/3-1 and BA 1477/15-1). V.S. acknowledges support of the EPSRC New Investigator Award grant EP/T027975/1.

\bibliographystyle{plain}

\end{document}